\documentclass[a4paper,12pt]{article}
\usepackage[top=3cm, left=2.5cm, bottom=3cm, right=2.5cm]{geometry}

\usepackage{geometry}
\usepackage{graphicx}
\usepackage{amsmath, amsthm, amssymb} 
\usepackage{overpic}
\usepackage{dsfont}
\usepackage{color}
\usepackage{tikz}

\RequirePackage[colorlinks, hyperindex, plainpages=false]{hyperref} 
\usepackage[ngerman, french, english]{babel}
\usepackage[latin1]{inputenc}
\usepackage[T1]{fontenc}
\usepackage{color}
\usepackage{setspace}
\usepackage{longtable}
\usepackage{bm}

\newtheorem{theorem}{Theorem}[section]
\newtheorem{proposition}[theorem]{Proposition}

\newtheorem{corollary}[theorem]{Corollary}

\newtheorem{lemma}[theorem]{Lemma}

\newtheorem{remark}[theorem]{Remark}

\newtheorem{assumption}[theorem]{Assumption}

\DeclareMathOperator\dist{dist}

\begin{document}

\title{Large Deviation Principle for Poisson driven SDEs in Epidemic Models.}
\author{
{Etienne Pardoux}\footnote{{Aix Marseille Universit\'e, CNRS, Centrale Marseille, I2M UMR 7373, 13453
Marseille, France; etienne.pardoux@univ-amu.fr; brice.samegni-kepgnou@univ-amu.fr.}}
\and
{Brice Samegni-Kepgnou}\setcounter{footnote}{6}$^\ast$
}

\maketitle

\begin{abstract} 
We consider a general class of epidemic models  obtained by applying the random time changes  of  \cite{ethier2009markov} to a collection of Poisson processes and we show the large deviation principle(LDP) for such models. We generalize to a more general situation the approach of followed by Dolgoashinnykh \cite{dolgoarshinnykhsample}  in the case of the SIR epidemic model. Thanks to an additional assumption which is satisfied in many examples, we simplify the recent work by P.Kratz and E.Pardoux \cite{Kratz2014}.
\end{abstract}

\vskip 3mm

\section*{Introduction}

In this paper, we are interested in a class of Poisson Models which arise in many fields such as chemical kinetics, ecological and epidemic models.  It is in fact a $d$ dimensional processes of the type
\begin{equation}\label{EqPoisson1}
Z^{N}(t):=Z^{N,z}(t):=\frac{[Nz]}{N} +\frac{1}{N} \sum_{j=1}^k h_{j} P_{j}\Big(\int_{0}^{t} N \beta_{j}(Z^N(s)) ds \Big).
\end{equation}
The components of the vector $Z^N(t)$ are the proportions of the population in the various compartments, and
 $(P_{j})_{1\le j\le k}$ are i.i.d. standard Poisson processes. The $h_{j}\in \mathds{Z}^{d}$ denote the $k$ distinct jump directions with jump rates $\beta_{j}(z)$ and $z\in A$,
where
\begin{equation}\label{defA}
 A=\Big\{z\in\mathbb{R}_{+}^{d}:\sum_{i=1}^{d}z_{i}\leq1\Big\}
\end{equation}
is the domain of the processes defined by \eqref{EqPoisson1}.

As we shall recall below, it is plain that under mild assumptions, as $N\to\infty$, $Z^{N}_{t}\to Y_{t}$ a.s., locally uniformly for $t>0$,
where $Y_t$ solves the ODE
\begin{equation*}
 \frac{dY_t}{dt}=b(Y_{t}), Y_{0}=z,
 \end{equation*}
where $b(z)=\sum_{j=1}^{k}\beta_{j}(z)h_{j}$. In this paper we want to investigate the large deviations from this law of large numbers.

Let us now be more precise about the initial condition $Z^{N}(0)=[Nz]/N$. In the models we have in mind, since each component of $Z^{N}(t)$ is a proportion in a population of total population size equal to $N$, we want $Z^N(t)$ to take its values in the set 
$A^{(N)}=\{z\in A,\ Nz\in\mathbb{Z}_+^d\}$. In particular, we want the initial condition $Z^{N}(0)$ to belong to this set $A^{(N)}$. If that is not the case, 
some of the components of the vector $Z^{N}(t)$ may become negative, while jumping from $a/N$ to $(a-1)/N$, $0<a<1$, which is not very natural. For that reason, we will use the following convention concerning the initial condition. We assume that there exists $z\in A$ such that for $1\le i\le d$, $N\ge1$,  $Z^{N}_{i}(0)=[Nz_{i}]/N$.

In all what follows, $D_{T,A}$ denotes the set of functions from $[0,T]$ into $A$ which are right continuous and have left limits and let $\mathcal{AC}_{T,A}$  be the subspace of absolutely continuous functions.  

We denote by $\mathcal{B}$  the Borel $\sigma$-field on $D_{T,A}$ and $\mathbb{P}^{N}_{z}$ the probability measure on paths whose initial condition is given by $Z^{N}(0)=[Nz]/N$ defined by 
 \begin{equation*}
 \mathbb{P}^{N}_{z}(B)=\mathbb{P}_{z}(Z^{N}\in B)\quad\forall B\in\mathcal{B}.
\end{equation*}
Our goal is to show that the probability measures $\mathbb{P}^{N}_{z}$, $N>1$, satisfy a large deviation principle with a good rate function $I_{T}$ that we define in subsection \ref{ratefunction}. In other words for any  $G$ open subset of $D_{T,A}$ and $F$ closed subset of $D_{T,A}$ we want to show the following inequalities:
\begin{align}
   -\inf_{\phi\in G}I_{T}(\phi) &\leq \liminf_{N\to\infty} \frac{1}{N}\log\mathbb{P}_{z}(Z^{N}\in G),\label{lbound}  \\
   & \limsup_{N\to\infty} \frac{1}{N}\log\mathbb{P}_{z}(Z^{N}\in F) \le -\inf_{\phi\in F}I_{T}(\phi). \label{ubound}
\end{align}

Large deviation principles is the subject of many treatises, see in particular \cite{Dembo2009}, \cite{Dupuis1997}, \cite{Feng2006}, \cite{Freidlin1998} and \cite{Shwartz1995}. Some of those books study large deviations for Poisson processes, like e.g. \cite{Shwartz1995}. However, in this treatise it is assumed that the rates of the Poisson processes are bounded away from zero, and hence their logarithms are bounded. The case of Poisson processes with vanishing rates is studied 
in \cite{shwartz2005large}. However their assumptions are not satisfied in our situation, as it is explained in  \cite{Kratz2014}. Our result have been already established in \cite{Kratz2014}. Our argument is simpler.
It is based upon an idea from \cite{dolgoarshinnykhsample} and forces us to add an assumption, which
is satisfied in all examples we have in mind.

That additional assumption is the following. We suppose that there exists a collection of mappings $\Phi_{a}:A\to A$, defined for each $a>0$, which are such that $z^{a}=\Phi_{a}(z)$ satisfies for each $a>0$
\begin{align*}
    |z-z^{a}|&\leq c_{1} a  \\
    \dist(z^{a},\partial A)&\geq c_{2}a  
\end{align*}
for some $0<\kappa_{2}<\kappa_{1}$. 
We now introduce the sets defined for all $a>0$ by
\begin{equation}\label{leB}
  B^{a}=\Big\{z\in A:\dist(z,\partial A)\geq c_{2}a\Big\}
\end{equation}
and
\begin{equation}\label{leR}
  R^{a}=\Big\{\phi\in\mathcal{AC}_{T,A}:\phi_{t}\in B^{a}\quad \forall t\in[0,T]\Big\}
\end{equation}
hence $\Phi_{a}$ maps $A$ into $B^{a}$.
\begin{remark}
Since our domain $A$ is convex, one can always define $\Phi_{a}=z+a(z_{0}-z)$, for some fixed $z_{0}\in \mathring{A}$. The same construction is possible for many non necessarily convex sets, provided $A$ is compact, and there is a point $z_{0}$ in its interior which is such that for each $z\in\partial A$, the segment joining $z_{0}$ and $z$  does not touch any other point of the boundary $\partial A$. We also note that for  such a choice of $\Phi_{a}$  and $A$ given by \eqref{defA} the constants $c_{1}$, $c_{1}$ can be defined by
\begin{align*}
c_{1}&=\sup_{z\in A}|z-z_{0}| \\
c_{2}&=\sin(\theta_{0})\inf_{z\in \partial A}|z-z_{0}|\leq\inf_{z\in \partial A}|z-z_{0}|\times\sin(\theta(z)).
\end{align*}
where $\theta(z)$ is the most acute angle between the boundary $\partial A$ and the vector $z_{0}-z$ and $\theta_{0}$ is a angle such that for all $z\in\partial A$, $\theta_{0}\leq\theta(z)\leq \pi/2$. For instance $\theta_{0}=\min_{1\leq\ell\leq6}\theta_{\ell}$.
\end{remark}

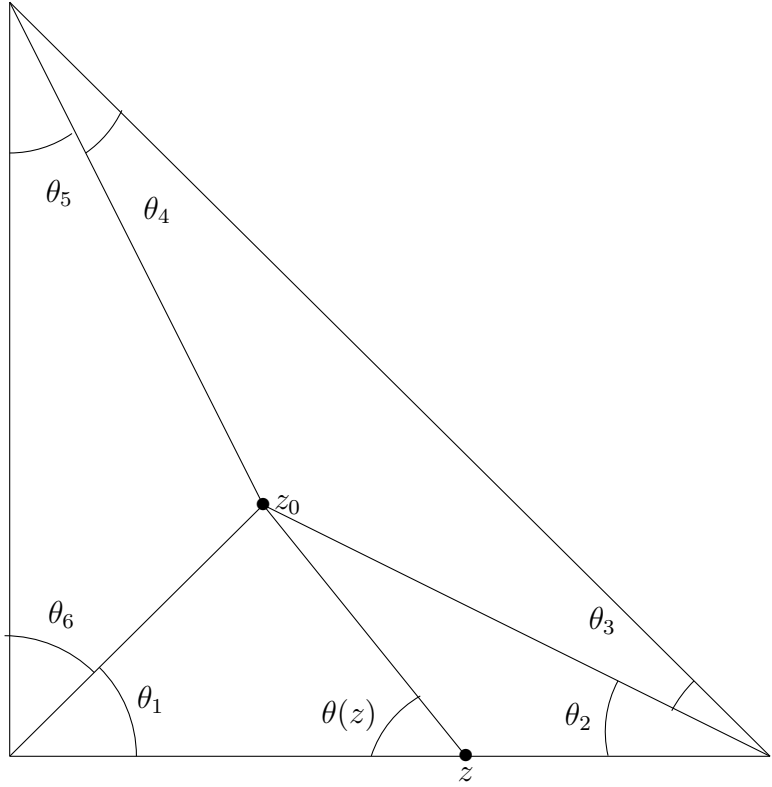
\begin{figure}
\begin{center}
\begin{tikzpicture}[scale=10]
\draw (0,0) -- (1,0);
\draw (0,0) -- (0,1);
\draw (1,0) -- (0,1);
\draw (0,0) -- (1/3,1/3);
\draw (1,0) -- (1/3,1/3);
\draw (0,1) -- (1/3,1/3);
\draw (1/3,1/3) node[right] {$z_{0}$};
\draw (1/3,1/3) node {$\bullet$};
\draw (1/6,0) arc (0:45:1/6) ;
\draw (1/9,1/9) arc (45:90:1/6); 
\draw (22:1/5) node {$\theta_{1}$};
\draw (70:1/5) node {$\theta_{6}$};
\draw (0.9,0.1) arc (135:152:1/6);
\draw (13:0.8) node {$\theta_{3}$};
\draw (0.8,0.1) arc (152:193:1/7);
\draw (4:0.75) node {$\theta_{2}$};
\draw (0,0.8) arc (270:305:1/7);
\draw (85:0.75) node {$\theta_{5}$};
\draw (0.1,0.8) arc (305:335:1/7);
\draw (75:0.75) node {$\theta_{4}$};
\draw (3/5,0) -- (1/3,1/3);
\draw (3/5,0) node[below] {$z$};
\draw (3/5,0) node {$\bullet$};
\draw (0.54,0.08) arc (120:162:1/7);
\draw (7:0.45) node {$\theta(z)$};
\end{tikzpicture}
\end{center}
\caption{Domain $A$}
\label{FigA}
\end{figure}

Moreover  for all $a>0$ we define
\begin{equation}\label{leC}
  C_{a}=\inf_{j}\inf_{z\in B^a}\beta_{j}(z).
\end{equation}
We remark that for all $a>0$,  $C_{a}>0$ and $\lim_{a\rightarrow 0}C_{a}=0$.

We make the following assumptions
\begin{assumption}\label{assump1}
\begin{enumerate}
 \item The rate functions $\beta_{j}$ are Lipschitz continuous with the Lipschitz constant equal to $C$.  \label{assump11}
 \item The $\beta_{j}$ are bounded by a positive constant $\sigma$. \label{assump12}
 \item There exist two constants $\lambda_{1}$ and $\lambda_{2}$ such that whenever $z\in A$ is such that $\beta_{j}(z)<\lambda_{1}$, $\beta_{j}(z^{a})>\beta_{j}(z)$ for all $a\in]0, \lambda_{2}[$ . \label{assump13}
 \item There exists constant $\nu\in]0,1/2[$  such that  \label{assump14}
 \begin{equation*}
  \lim_{a\rightarrow 0}a^{\nu}\log C_{a}=0.
 \end{equation*}
 This means in particular that there exists $a_{0}>0$ such that for all $a<a_{0}$, $C_{a}\geq e^{-a^{-\nu}}$.
 \end{enumerate}
\end{assumption}
Let us comment on Assumption \ref{assump1}. Assumption \ref{assump1}.\ref{assump11} is quite standard and ensures in particular that the ODE \eqref{ODE} admits a unique solution. For the compartmental epidemiological models we consider, this assumption is always true because the $\beta_{j}(z)$ are usually polynomials and $A$ is compact. Also the assumption  \ref{assump1}.\ref{assump12} is always true because the domain of our process is compact. Assumption  \ref{assump1}.\ref{assump13} will follow from the fact that close to the boundary, "small" rates are increasing when we follow a direction towards the inside of the domain. Concerning the assumption  \ref{assump1}.\ref{assump14}, such an assumption is true for the models we study because the rates are usually polynomials.

For all $\phi, \psi \in D_{T,A}$ we will define the distance between $\phi$ and $\psi$ by
\begin{equation*}
 \|\phi-\psi\|_{T}=\sup_{t\leq T}|\phi_{t}-\psi_{t}|
\end{equation*}
where $|.|$ denotes the Euclidean norm in $\mathbb{R}^{d}$. 

The remainder of this paper is structured as follows. In section 1, we formulate the law of large numbers, we define a good rate function for our large deviation principle and we establish some properties that it satisfies. The second section concerns the proof of the lower bound \eqref{lbound} and the third one the proof of the upper bound \eqref{ubound}. The last section of this paper states a result concerning the asymptotic behavior of the exit time from the domain of attraction of a stable point for the dynamical system \eqref{ODE} as well as the exponential asymptotic of its mean $\mathbb{E}_{z}(\tau^{N}_{O})$. For epidemic models, this exit time is the time of extinction of an endemic disease.

\section{Somes Important Results} \label{secresults}

\subsection{Law of Large Number and Change of Measure}

We now prove the law of large number.
\begin{theorem}\label{LLN}
Let $Z^{N,z}(t)$  the solution of stochastic differential equation Poissonian \eqref{EqPoisson1} with an initial condition $[Nz]/N$. Assume that  the assumption \ref{assump1}.\ref{assump11} holds.  Then 
\begin{equation*}
  \lim_{N\rightarrow \infty}\|Z^{N,z}-Y^{z}\|_{T}=0\quad\text{a.s}.
 \end{equation*}
 Where $Y^{z}(.)$ is the solution of the ODE
 \begin{equation}\label{ODE}
\frac{dY^{z}(t)}{dt}:=b(Y^{z}(t)) 
\end{equation}
with an initial condition $z$ and where
\[
b(z):=\sum_{j=1}^k \beta_{j}(z) h_{j}.
\]
\end{theorem}

\begin{proof}
By using the Lipschitz continuity of  $b$, we have with $M_{j}(t)=P_{j}(t)-t$, $\widetilde{M}^{N}_{j}(t)=\frac{1}{N}P_{j}(N.t)-t$
\begin{align}\label{apllygron}
    & |Z^{N}(t)-Y^{z}(t)| \leq \Big|\frac{[Nz]}{N}-z\Big|+\int_{0}^{t}|b(Z^{N}(s))-b(Y^{z}(s))|ds+\frac{1}{N}\Big|\sum_{j=1}^{k}h_{j}M_{j}\Big(N\int_{0}^{t}\beta_{j}(Z^{N}(s))ds\Big)\Big| \nonumber \\
    &  \leq \Big|\frac{[Nz]}{N}-z\Big|+kC\sqrt{d}\int_{0}^{t}|Z^{N}(s)-Y^{z}(s)|ds+\sqrt{d}\sum_{j=1}^{k}\Big|\widetilde{M}^{N}_{j}\Big(\int_{0}^{t}\beta_{j}(Z^{N}(s))ds\Big)\Big| \nonumber \\
    & \leq \Big|\frac{[Nz]}{N}-z\Big|+kC\sqrt{d}\int_{0}^{t}|Z^{N}(s)-Y^{z}(s)|ds+k\sqrt{d}\sup_{j}\sup_{t\leq T}\Big|\widetilde{M}^{N}_{j}\Big(\int_{0}^{t}\beta_{j}(Z^{N}(s))ds\Big)\Big| 
\end{align}
Let $\xi^{N}_{j}(t)=\Big|\widetilde{M}^{N}_{j}\Big(\int_{0}^{t}\beta_{j}(Z^{N}(s))ds\Big)\Big|$. From the strong law of large numbers for a Poisson process, 
we have for all $j=1,...,k$
\begin{equation*}
\frac{P_{j}(Nt)}{N}\to t\quad a.s.\quad\text{as}\quad N\to\infty.
\end{equation*}
As we have pointwise convergence of a sequence of increasing function towards a continuous function we can use the second Dini theorem to conclude that this convergence is uniform on any compact time interval, hence for  $0\leq v<\infty$ and $j=1,...k$
\begin{equation*}
  \lim_{N\rightarrow \infty}\sup_{u\leq v}|\widetilde{M}^{N}_{j}(u)|=0\quad\text{a.s}.
 \end{equation*}

 As the $\beta_{j}$ are bounded by $\sigma$, it follows that
\begin{equation*}
  \lim_{N\rightarrow \infty}\sup_{t\leq T}\xi^{N}_{j}(t)=0\quad\text{a.s}.
 \end{equation*} 
 for $j=1,...,k$. \\ By using by Gronwall's inequality stated above we have
 \begin{equation*}
   |Z^{N}_{t}-Y^{z}_{t}| \leq k\sqrt{d}\Big(\Big|\frac{[Nz]}{N}-z\Big|+\sup_{j}\sup_{t\leq T}\xi^{N}_{j}(t)\Big)\exp\{kC\sqrt{d} t\}
 \end{equation*} 
 and the result follows.
 \end{proof}

We shall need the following Girsanov theorem . Let $Q$ equal to the random number of jumps of $Z^{N}$ in the interval $[0,T]$, $\tau_{p}$  be the time of the $p^{th}$  jump for $p=1,...,Q$ and define
\begin{equation*}
\delta_{p}(j)=
\begin{cases}
      1& \text{if the $p^{th}$ jump is in the direction $h_{j}$ }, \\
      0& \text{otherwise}.
\end{cases}
\end{equation*}
We shall denote $\mathcal{F}^{N}_{t}=\sigma\{Z^{N}(s), 0\leq s\leq t\}$. Consider another set of rates $\tilde{\beta}_{j}(z)$, 
$1\leq j\leq k$. Combining Theorem  VI T3 from \cite{bremaud} and Theorem 2.4 from \cite{sokol}, we have

\begin{theorem}\label{girsanov}
Let $\widetilde{\mathbb{P}}^{N}$ denote the law of $Z^{N}$ when the rates are rates $\tilde{\beta}_{j}(.)$.  Then provided that 
$\sup_{z\in A}\frac{\tilde{\beta}_j(z)}{\beta_j(z)}<\infty$, which implies in particular that $\{z:\beta_{j}(z)=0\}\subset\{z:\tilde{\beta}_{j}(z)=0\}$, on the $\sigma$-algebra $\mathcal{F}^{N}_{t}$, $\widetilde{\mathbb{P}}^{N}\big|_{\mathcal{F}^{N}_{T}}<<\mathbb{P}^{N}\big|_{\mathcal{F}^{N}_{T}}$, and
\begin{align}\label{likelihood}
\xi_{T}&=\xi^{N}_{T}= \frac{d\widetilde{\mathbb{P}}^{N}\big|_{\mathcal{F}^{N}_{T}}}{d\mathbb{P}^{N}\big|_{\mathcal{F}^{N}_{T}}}  \nonumber\\
&=\left(\prod_{p=1}^{Q}\prod_{j=1}^{k}\left[\frac{\tilde{\beta}_{j}(Z^{N}(\tau_{p}^{-}))}{\beta_{j}(Z^{N}(\tau_{p}^{-}))}\right]^{\delta_{p}(j)}\right)  \exp\Big\{N\sum_{j=1}^{k}\int_{0}^{T}(\beta_{j}(Z^{N}(t))-\tilde{\beta}_{j}(Z^{N}(t)))dt\Big\}.
\end{align}
\end{theorem}
\begin{corollary}\label{cogirsanov}
For all non-negative measurable function $X\geq0$,
\begin{equation*}
\mathbb{E}(X)\geq\tilde{\mathbb{E}}(\xi^{-1}_{T}X)
\end{equation*}
\end{corollary}
\begin{proof}
As $X\geq0$, we write
\begin{equation*}
\mathbb{E}(X)\geq\mathbb{E}(X\mathbf{1}_{\{\xi_{T}\neq0\}})=\tilde{\mathbb{E}}(\xi^{-1}_{T}X\mathbf{1}_{\{\xi_{T}\neq0\}})=\tilde{\mathbb{E}}(\xi^{-1}_{T}X).
\end{equation*}
This last equality comes from the fact that $\widetilde{\mathbb{P}}(\xi_{T}=0)=0$ i.e. $\xi^{-1}_{T}$ is well-defined $\widetilde{\mathbb{P}}-$almost surely.
\end{proof}

\subsection{The Rate Function}\label{ratefunction}

For all $\phi\in\mathcal{AC}_{T,A}$, let $\mathcal{A}_{d}(\phi)$ the set of vector valued Borel measurable functions $\mu$ such that for all $j=1,...,k$, $\mu^{j}_{t}\geq0$ and 
\begin{equation*}\label{allowed}
  \frac{d\phi_{t}}{dt}=\sum_{j=1}^{k} \mu^{j}_{t} h_{j}, \quad\text{t a.e}.
\end{equation*}
We define the rate function
\begin{equation*}
I_{T}(\phi) :=
 \begin{cases}
\inf_{\mu\in\mathcal{A}_{d}(\phi)}I_{T}(\phi|\mu),& \text{ if } \phi\in\mathcal{AC}_{T,A}; \\
\infty ,& \text{ else.}
\end{cases}
\end{equation*}
where
\begin{equation*}
  I_{T}(\phi|\mu)=\int_{0}^{T}\sum_{j=1}^{k}f(\mu_{t}^{j},\beta_{j}(\phi_{t}))dt
\end{equation*}
with $f(\nu,\omega)=\nu\log(\nu/\omega)-\nu+\omega$.
 We assume in the definition of $f(\nu,\omega)$ that for all $\nu>0$, $\log(\nu/0)=\infty$ and $0\log(0/0)=0\log(0)=0$.

 By using the Legendre-Fenchel transform we define another rate function by
 \begin{equation*}
\tilde{I}_{T}(\phi) :=
 \begin{cases}
\int_{0}^{T}L(\phi_{t},\phi'_{t})dt& \text{ if } \phi\in\mathcal{AC}_{T,A} \\
\infty & \text{ else.}
\end{cases}
\end{equation*}
where for all $z\in A$, $y\in\mathbb{R}^{d}$ 
\begin{equation*}
L(z,y)=\sup_{\theta\in\mathbb{R}^{d}}\ell(z,y,\theta)
\end{equation*}
with for all $z\in A$, $y\in\mathbb{R}^{d}$ and $\theta\in\mathbb{R}^{d}$
\begin{equation*}
\ell(z,y,\theta)=\big<\theta,y\big>-\sum_{j=1}^{k}\beta_{j}(z)(e^{\big<\theta,h_{j}\big>}-1)
\end{equation*}

We now show the equality between these two definitions of the rate function.
\begin{lemma}
For all  $\phi\in\mathcal{AC}_{T,A} $ and $\mu\in\mathcal{A}_{d}(\phi)$ we have
\begin{equation*}
\tilde{I}_{T}(\phi)\leq I_{T}(\phi|\mu).
\end{equation*}
In particular $\tilde{I}_{T}(\phi)\leq I_{T}(\phi)$
\end{lemma}

\begin{proof}
Assume first that for some $B\in\mathcal{B}([0,T])$, with $\int_{B}dt>0$ such that for all $t\in B$ there exists $1\leq j\leq k$ such that $\mu_{t}^{j}>0$ and $\beta_{j}(\phi_{t})=0$ then $I_{T}(\phi|\mu)=\infty$ and the inequality is true. We now assume that for almost all $t\in[0,T]$ and for all 
$j\in{1,...,k}$ , $\mu_{t}^{j}>0$ only if $\beta_{j}(\phi_{t})>0$ then for all $\theta\in\mathbb{R}^{d}$
\begin{align*}
    \ell(\phi_{t},\phi_{t},\theta)&= \sum_{j=1}^{k}\mu_{t}^{j}\big<\theta,h_{j}\big>-\beta_{j}(x)(e^{\big<\theta,h_{j}\big>}-1) \\
    & = \sum_{j=1}^{k}g_{\mu_{t}^{j},\beta_{j}(\phi_{t})}(\big<\theta,h_{j}\big>) \\
    &\leq \sum_{j=1}^{k}g_{\mu_{t}^{j},\beta_{j}(\phi_{t})}\Big(\log\frac{\mu_{t}^{j}}{\beta_{j}(\phi_{t})}\Big) \\
    &= \sum_{j=1}^{k} f(\mu_{t}^{j},\beta_{j}(\phi_{t})),   
    \end{align*}
    since $g_{\nu,\beta}(z)=\nu z-\beta(e^{z}-1)$ is a function which achieves its maximum at $z=\log\frac{\nu}{\beta}$.
\end{proof}

\begin{lemma}
For all $\phi\in\mathcal{AC}_{T,A} $, 
\begin{equation*}
I_{T}(\phi)\leq\tilde{I}_{T}(\phi).
\end{equation*}
\end{lemma}
\begin{proof}
If  $\tilde{I}_{T}(\phi)=\infty$ the inequality is true. We now assume that $\tilde{I}_{T}(\phi)<\infty$ 
then for almost all $t\in[0,T]$ we have $L(\phi_{t},\phi'_{t})=\sup_{\theta\in\mathbb{R}^{d}}\ell(\phi_{t},\phi'_{t},\theta)<\infty$
then by \cite{Kratz2014} there exists a maximizing sequence $(\theta_{n})_{n}$ of $\ell(\phi_{t},\phi'_{t},.)$ namely $L(\phi_{t},\phi'_{t})=\lim_{n}\ell(\phi_{t},\phi'_{t},\theta_{n})$ and constants $s_{j}$ such that for all $j=1,...,k$,
\begin{equation*}
\lim_{n}\exp\{\big<\theta_{n},h_{j}\big>\}=s_{j}.
\end{equation*}
Then we have 
\begin{equation*}
\lim_{n}\big<\theta_{n},\phi'_{t}\big>=L(\phi_{t},\phi'_{t})+\sum_{j:\beta_{j}(\phi_{t})>0}\beta_{j}(\phi_{t})(s_{j}-1).
\end{equation*}
Moreover we differentiate with respect to $\theta$ and obtain for all $n$
\begin{equation*}
\nabla_{\theta} \ell(\phi_{t},\phi'_{t},\theta_{n})=\frac{d\phi_{t}}{dt}-\sum_{j:\beta_{j}(\phi_{t})>0}\beta_{j}(\phi_{t})h_{j}\exp\{\big<\theta_{n},h_{j}\big>\}.
\end{equation*}
As  $(\theta_{n})_{n}$ is a maximizing sequence we have for all $t\in[0,T]$
\begin{equation*}
\lim_{n}\nabla_{\theta} \ell(\phi_{t},\phi'_{t},\theta_{n})=\frac{d\phi_{t}}{dt}-\sum_{j:\beta_{j}(\phi_{t})>0}\beta_{j}(\phi_{t})s_{j}h_{j}=0.
\end{equation*}
Thus, for almost all $t\in[0,T]$
\begin{equation*}
\frac{d\phi_{t}}{dt}=\sum_{j=1}^{k}\beta_{j}(\phi_{t})s_{j}h_{j}=\sum_{j=1}^{k}\mu^{j}_{t}h_{j}.
\end{equation*}
Where for almost all $t\in[0,T]$ and $j=1,...,k$ 
\begin{equation*}
\mu^{j}_{t}=\beta_{j}(\phi_{t})s_{j}.
\end{equation*}
We deduce that
\begin{align*}
    & I_{T}(\phi)\leq   I_{T}(\phi|\mu)\\
    &  =\int_{0}^{T}\sum_{j=1}^{k}f(\mu_{t}^{j},\beta_{j}(\phi_{t}))dt \\
    &=\int_{0}^{T}\sum_{j=1}^{k}\big\{\mu^{j}_{t}\log s_{j}+\beta_{j}(\phi_{t})(1-s_{j})\big\}dt \\
    &=\int_{0}^{T}L(\phi_{t},\phi'_{t})dt =\tilde{I}_{T}(\phi).
   \end{align*}
\end{proof}
\par The proof of the following theorem can be found in \cite{Kratz2014}.
\begin{theorem}
$I_{T}=\tilde{I}_{T}$ is a good rate function.
\end{theorem}
\begin{proof}
As the $\beta_{j}$ are bounded and continuous, we deduce from Lemma 4.20 in \cite{Kratz2014} that $\tilde{I}_{T}$ is lower semicontinuous with respect to Skorokhod's metric on $D_{T,A}$. Therefore the level set $\Phi(s)=\{\phi\in D_{T,A}:\tilde{I}_{T}(\phi)\leq s\}$ are closed and one can show that those sets are equicontinuous. We also know that $A$ is compact and then the relatively compact subsets of $C([0,T],A)$  are exactly the subsets of equicontinuous functions. Thus the level sets $\Phi(s)$ are compact since they are closed and relatively compact.
\end{proof}

The following result is a direct consequence of Lemma 4.22 in \cite{Kratz2014}
\begin{lemma}\label{semiconz}
Let $F$ a closed subset of $D_{T,A}$ and $z\in A$. We have
\begin{equation*}
\lim_{\epsilon\to0}\inf_{y\in A, |y-z|<\epsilon}\inf_{\phi\in F, \phi_{0}=y}I_{T}(\phi)=\inf_{\phi\in F, \phi_{0}=z}I_{T}(\phi).
\end{equation*}
\end{lemma}

\begin{lemma}\label{le10}
Let $s>0$, $\phi\in D_{T,A}$ and $\mu\in\mathcal{A}_{d}(\phi)$ such that $I_{T}(\phi|\mu)\leq s$ then for all $0\leq t_{1}, t_{2}\leq T$ such that $t_{2}-t_{1}\leq1/\sigma$,
\begin{equation*}
  \int_{t_{1}}^{t_{2}}\mu^{j}_{t}dt\leq\frac{s+1}{-\log(\sigma(t_{2}-t_{1}))}\quad\forall j=1,...,k.
\end{equation*}
\end{lemma}
\begin{proof}
We have 
\begin{equation*}
  \int_{0}^{T}f(\mu^{j}_{t},\beta_{j}(\phi_{t}))dt\leq I_{T}(\phi|\mu)\leq s.
\end{equation*}
moreover, the function $h(x)=x\log(x/\sigma)-x$ is convex in $x$ so that for all $0\leq t_{1}, t_{2}\leq T$
\begin{align*}
  &h\Big(\frac{1}{t_{2}-t_{1}}\int_{t_{1}}^{t_{2}}\mu^{j}_{t}dt\Big)
  \leq \frac{1}{t_{2}-t_{1}}\int_{t_{1}}^{t_{2}}h(\mu^{j}_{t})dt\\
  &\leq\frac{1}{t_{2}-t_{1}}\int_{t_{1}}^{t_{2}}\Big(\mu^{j}_{t}\log\frac{\mu^{j}_{t}}{\beta_{j}(\phi_{t})}-\mu^{j}_{t}+\beta_{j}(\phi_{t})\Big)dt\\
  &\leq\frac{s}{t_{2}-t_{1}}.
\end{align*}
It is easy to show that for all $\alpha>0$, $h(x)\geq\alpha x-\sigma\exp\{\alpha\}$ and then for all $\alpha>0$
\begin{equation*}
  \int_{t_{1}}^{t_{2}}\mu^{j}_{t}dt\leq\frac{1}{\alpha}(s+(t_{2}-t_{1})\sigma\exp\{\alpha\}).
\end{equation*}
Therefore If $t_{2}-t_{1}<1/\sigma$ taking $\alpha=-\log(\sigma(t_{2}-t_{1}))$, the result follows.
\end{proof}

For $\phi\in D_{T,A}$ let $\phi^{a}$ defined by $\phi^{a}_{t}=(1-a)\phi_{t}+az_{0}$  and
we have $\phi^{a}\in R^{a}$.

\begin{lemma}\label{le11}
For all $\phi\in D_{T,A}$ we have $\limsup_{a\rightarrow0}I_{T}(\phi^{a})\leq I_{T}(\phi)$.
\end{lemma}
\begin{proof}
First if $I_{T}(\phi)=\infty$ the result is easy.  If $I_{T}(\phi)<\infty$, $\forall\eta>0$ there exists $\mu$ such that $I_{T}(\phi|\mu)\leq I_{T}(\phi)+\eta$. Let $\mu^{a}=(1-a)\mu$ then $\mu^{a}\in\mathcal{A}_{d}(\phi^{a})$. We will now show that
\begin{equation}\label{conphia}
I_{T}(\phi^{a}|\mu^{a})\rightarrow I_{T}(\phi|\mu)\quad\text{as}\quad a\rightarrow0,
\end{equation}
which clearly implies the result since
\begin{align*}
\limsup_{a\rightarrow0}I_{T}(\phi^{a})&\leq\limsup_{a\rightarrow0}I_{T}(\phi^{a}|\mu^{a}) \\
&=I_{T}(\phi|\mu)\leq I_{T}(\phi)+\eta.
\end{align*}
 By the convexity of $f(\nu,\omega)$ in $\nu$ and because $0\leq\mu^{j,a}_{t}\leq\mu^{j}_{t}$, we have
\begin{align*}
  &0\leq f(\mu^{j,a}_{t},\beta_{j}(\phi^{a}_{t}))\leq f(0,\beta_{j}(\phi^{a}_{t})) +f(\mu^{j}_{t},\beta_{j}(\phi^{a}_{t})) \\
  &\leq\sigma+f(\mu^{j}_{t},\beta_{j}(\phi^{a}_{t})).
\end{align*}
Moreover we have
\begin{align*}
  &f(\mu^{j}_{t},\beta_{j}(\phi^{a}_{t}))=\mu^{j}_{t}\log\frac{\mu^{j}_{t}}{\beta_{j}(\phi_{t}^{a})}-\mu_{t}^{j}+\beta_{j}(\phi_{t}^{a})\\
  &=\mu^{j}_{t}\log\frac{\mu^{j}_{t}}{\beta_{j}(\phi_{t})}-\mu^{j}_{t}+\beta_{j}(\phi_{t})+\mu^{j}_{t}\log\frac{\beta_{j}(\phi_{t})}{\beta_{j}(\phi_{t}^{a})} +\beta_{j}(\phi^{a}_{t})-\beta_{j}(\phi_{t})\\
  &\leq f(\mu^{j}_{t},\beta_{j}(\phi_{t}))+\sigma+\mu^{j}_{t}\log\frac{\beta_{j}(\phi_{t})}{\beta_{j}(\phi_{t}^{a})}.
\end{align*}
If $\beta_{j}(\phi_{t})<\lambda_{1}$ then $\beta_{j}(\phi_{t})\leq\beta_{j}(\phi^{a}_{t})$  and $\log\frac{\beta_{j}(\phi_{t})}{\beta_{j}(\phi^{a}_{t})}<0$.
\\ If $\beta_{j}(\phi_{t})\geq\lambda_{1}$ then using the Lipschitz continuity of the rates $\beta_{j}$ we have
\begin{align*}
  &\log\frac{\beta_{j}(\phi_{t})}{\beta_{j}(\phi^{a}_{t})}\leq\log\frac{\beta_{j}(\phi_{t})}{\beta_{j}(\phi_{t})-C c_{1}a}\leq\log\frac{\lambda_{1}}{\lambda_{1}-C c_{1}a} \\ 
  &\leq\log\frac{1}{1-C c_{1}a/\lambda_{1}}<\frac{2C c_{1}a}{\lambda_{1}}<\frac{2C c_{1} c_{2}}{\lambda_{1}}.
\end{align*}
Since $\log(1/(1-x))<2x$ for $0<x<1/2$; here, we take $a$ small enough to ensure $C c_{1}a<\lambda_{1}/2$.
Finally for all $a<(\lambda_{1}/2c_{1}C)\wedge\lambda_{2}$
\begin{equation*}
  0\leq f(\mu^{j,a}_{t},\beta_{j}(\phi^{a}_{t}))\leq f(\mu^{j}_{t},\beta_{j}(\phi_{t}))+2\sigma+\frac{2C c_{1}\lambda_{2}}{\lambda_{1}}\mu^{j}_{t}.
\end{equation*}
By Lemma \ref{le10} $\mu^{j}_{t}$ is integrable, we have bounded $f(\mu^{j,a}_{t},\beta_{j}(\phi^{a}_{t}))$ for $0<a<(\lambda_{1}/2 c_{1}C)\wedge\lambda_{2}$ by an integrable function. Moreover $f(\mu^{j,a}_{t},\beta_{j}(\phi^{a}_{t}))\rightarrow f(\mu^{j}_{t},\beta_{j}(\phi_{t}))$ since first $I_{T}(\phi)<\infty$ means that for almost all $t\in[0,T]$ and $1\leq j\leq k$, $\mu^{j}_{t}>0$ only if $\beta_{j}(\phi_{t})>0$ and then 
\begin{align*}
    |f(\mu^{j,a}_{t},\beta_{j}(\phi^{a}_{t}))-f(\mu^{j}_{t},\beta_{j}(\phi_{t}))|&\leq(1-a)\mu^{j}_{t}\log(1-a)+|\beta_{j}(\phi^{a}_{t})-\beta_{j}(\phi_{t})|   \\
    &+|(1-a)\mu^{j}_{t}-\mu^{j}_{t}| +\Big|(1-a)\mu^{j}_{t}\log\frac{\mu^{j}_{t}}{\beta_{j}(\phi^{a}_{t})}-\mu^{j}_{t}\log\frac{\mu^{j}_{t}}{\beta_{j}(\phi_{t})}\Big|. 
\end{align*}
The last term of this inequality is either $0$ or converge to $0$ when $a$ tend to $0$. We deduce from, the dominated convergence theorem that
\begin{equation*}
  \int_{0}^{T}f(\mu^{j,a}_{t},\beta_{j}(\phi^{a}_{t}))dt\rightarrow \int_{0}^{T}f(\mu^{j}_{t},\beta_{j}(\phi_{t}))dt\quad\text{as}\quad a\rightarrow0,
\end{equation*}
from which \eqref{conphia} follows, hence the result.
\end{proof}

\begin{lemma}\label{le12}
Let $a>0$ and $\phi\in R^{a}$ such that $I_{T}(\phi)<\infty$. For all $\eta>0$ there exists $L>0$ and $\phi^{L}\in R^{a/2}$ such that $\|\phi-\phi^{L}\|_{T}<c_{1}\frac{a}{2}$ and $I_{T}(\phi^{L}|\mu^{L})\leq I_{T}(\phi)+\eta$ where $\mu^{L}\in\mathcal{A}_{d}(\phi^{L})$ such that $\mu^{L,j}_{t}<L$, $j=1,...,k$.
\end{lemma}

\begin{proof}
Let $\eta>0$ and $\mu\in\mathcal{A}_{d}(\phi)$ such that  $I_{T}(\phi|\mu)<I_{T}(\phi)+\eta/2$.
For $L>0$ let $\mu_{t}^{L,j}=\mu_{t}^{j}\wedge L$ and let $\phi^{L}$ a solution of the ODE
\begin{equation*}
  \frac{d\phi^{L}_{t}}{dt}=\sum_{j=1}^{k}\mu^{L,j}_{t}h_{j}.
\end{equation*}
We first show that for $L$ sufficiently large $\phi^{L}$ is close to $\phi$ in supnorm.
Since $\mu^{j}_{t}$ is integrable over $[0,T]$ and $0\leq\mu^{L,j}_{t}\leq\mu^{j}_{t}$, the monotone convergence theorem implies that there exists $L_{a}>0$ such that for all $L>L_{a}$, $j=1,...,k$
\begin{equation*}
  \int_{0}^{T}\big|\mu^{L,j}_{t}-\mu^{j}_{t}\big|dt<\epsilon_{a}=c_{2}\frac{a}{2k\sqrt{d}}.
\end{equation*}
We deduce that
\begin{equation*}
  |\phi^{L,i}_{t}-\phi^{i}_{t}|\leq\sum_{j=1}^{k}|h^{i}_{j}|\int_{0}^{T}\big|\mu^{L,j}_{t}-\mu^{j}_{t}\big|dt<k\epsilon_{a}
\end{equation*}
and then we have for all $L>L_{a}$
$\|\phi^{L}-\phi\|<c_{1}\frac{a}{2}$ since $c_{2}<c_{1}$.
As $\phi\in R^{a}$ the above also ensures that $\phi^{L}\in R^{a/2}$ since for all $t\in[0,T]$
\begin{align*}
\dist(\phi^{L}_{t},\partial A)&\geq \dist(\phi_{t},\partial A)-|\phi^{L}_{t}-\phi_{t}| \\
&\geq c_{2}a-c_{2}\frac{a}{2}= c_{2}\frac{a}{2}.
\end{align*}
To show the convergence of $I_{T}(\phi^{L}|\mu^{L})$ to $I_{T}(\phi|\mu)$ we need to remark first using the convexity of $f(\nu,\omega)$ in $\nu$ that we have
\begin{equation*}
  f(\mu^{L,j}_{t},\beta_{j}(\phi^{L}_{t}))\leq f(0,\beta_{j}(\phi^{L}_{t}))+f(\mu^{j}_{t},\beta_{j}(\phi^{L}_{t})).
\end{equation*}
Since $\phi\in R^{a}$, $C_{a}\leq\beta_{j}(\phi_{t})\leq\sigma$ and $C_{a/2}\leq\beta_{j}(\phi^{L}_{t})\leq\sigma$ for all $L>L_{a}$,
notice that
\begin{equation*}
\frac{\partial f(\nu,\omega)}{\partial\omega}=-\frac{\nu}{\omega}+1
\end{equation*}
and therefore on the interval $[K_{a}, \theta]$ where $K_{a}=C_{a}\wedge C_{a/2}$
\begin{equation*}
  |f(\mu_{t}^{j},\beta_{j}(\phi^{L}_{t}))-f(\mu_{t}^{j},\beta_{j}(\phi_{t}))|<\bar{C}(\mu_{t}^{j}+1)
\end{equation*}
for some constant $\bar{C}>0$. Since $\mu_{t}^{j}$ and $f(\mu_{t}^{j},\beta_{j}(\phi_{t}))$ are integrable the dominated convergence theorem implies that
\begin{equation*}
  \int_{0}^{T}f(\mu^{L,j}_{t},\beta_{j}(\phi^{L}_{t}))dt\rightarrow\int_{0}^{T}f(\mu_{t}^{j},\beta_{j}(\phi_{t}))dt~~ \text{as}~~ L\to\infty.
\end{equation*}
\end{proof}

 Let $\epsilon>0$ be such that $T/\epsilon\in\mathds{N}$ and let the $\phi^{\epsilon}$ be the polygonal approximation of $\phi$ defined  for $t\in[\ell\epsilon,(\ell+1)\epsilon)$ by
\begin{equation}\label{poly1}
  \phi^{\epsilon}_{t}=\phi_{\ell\epsilon}\frac{(\ell+1)\epsilon-t}{\epsilon}+\phi_{(\ell+1)\epsilon}\frac{t-\ell\epsilon}{\epsilon}.
\end{equation}

\begin{lemma}\label{le13}
For any $\eta>0$. Let $0<a<1$, $\phi\in R^{a}$ and $\mu\in\mathcal{A}_{d}(\phi)$ such that $\mu^{j}_{t}<L$, $j=1,...,k$ for some $L>0$ and $I_{T}(\phi|\mu)<\infty$ then there exists $a_{\eta}$ such that for all $a<a_{\eta}$ there exists an $\epsilon_{a}>0$ such that for all $\epsilon<\epsilon_{a}$ the polygonal approximation $\phi^{\epsilon}\in R^{a/2}$ and $\|\phi-\phi^{\epsilon}\|_{T}<c_{2}\frac{a}{2}<c_{1}\frac{a}{2}$. 
Moreover, there exists $\mu^{\epsilon}\in\mathcal{A}_{d}(\phi^{\epsilon})$ such that $\mu^{\epsilon,j}_{t}<L$, $j=1,...,k$ and $I_{T}(\phi^{\epsilon}|\mu^{\epsilon})\leq I_{T}(\phi|\mu)+\eta$.
\end{lemma}
\begin{proof}
Since $\phi$ is uniformly continuous on $[0,T]$ there exists an $\epsilon_{a}$ such that $\forall\epsilon<\epsilon_{a}$
\begin{equation*}
  \sup_{|t-t'|<2\epsilon}|\phi_{t}-\phi_{t'}|<c_{2}\frac{a e^{-a^{-\nu}}}{4}
\end{equation*}
and then  $\|\phi-\phi^{\epsilon}\|_{T}<c_{2}\frac{a}{2}$ and $\phi^{\epsilon}\in R^{a/2}$ since for all $t\in[0,T]$,
\begin{align*}
    \dist(\phi^{\epsilon}_{t},\partial A)&\geq  \dist(\phi_{t},\partial A)-|\phi^{\epsilon}_{t}-\phi_{t}|  \\
    &\geq  \dist(\phi_{t},\partial A)-|\phi_{\ell\epsilon}-\phi_{t}|-|\phi_{(\ell+1)\epsilon}-\phi_{t}|\geq c_{2}\frac{a}{2}.
\end{align*}
For $t\in]\ell\epsilon,(\ell+1)\epsilon[$
\begin{equation*}
  \frac{d\phi^{\epsilon}_{t}}{dt}=\frac{\phi_{(\ell+1)\epsilon}-\phi_{\ell\epsilon}}{\epsilon}=\frac{1}{\epsilon}\sum_{j=1}^{k}h_{j}\int_{\ell\epsilon}^{(\ell+1)\epsilon}\mu^{j}_{t}dt
\end{equation*}
therefore for all $t\in[\ell\epsilon, (\ell+1)\epsilon[$, $\mu^{\epsilon}_{t}$  defined by
\begin{equation*}
  \mu^{\epsilon,j}_{t}=\frac{1}{\epsilon}\int_{\ell\epsilon}^{(\ell+1)\epsilon}\mu^{j}_{t}dt, j=1,...,k
\end{equation*}
is such that $\mu^{\epsilon}\in\mathcal{A}_{d}(\phi^{\epsilon})$ and is constant over $[\ell\epsilon, (\ell+1)\epsilon[$. We also note that $\mu^{\epsilon,j}_{t}\leq L$ for all $j=1,...,k$. Moreover if $0<\nu\leq L$ and $\omega\geq C_{a}$ then
\begin{equation*}
  \Big|\frac{\partial f(\nu,\omega)}{\partial\omega}\Big|=|-\frac{\nu}{\omega}+1|\leq\frac{L}{C_{a}}+1.
\end{equation*}
By the assumption \ref{assump1} \ref{assump14}, there exists $\tilde{a}_{\eta}>0$ such that for all $a<\tilde{a}_{\eta}$ 
\begin{equation*}
\frac{L}{C_{a}}+1\leq L e^{a^{-\nu}}+1
\end{equation*}

Then for $t\in[\ell\epsilon,(\ell+1)\epsilon[$  and $a<\bar{a}_{\eta}$, $\tilde{a}_{\eta}$
\begin{align*}
  &|f(\mu^{\epsilon,j}_{t},\beta_{j}(\phi_{t}^{\epsilon}))-f(\mu^{\epsilon,j}_{t},\beta_{j}(\phi_{\ell\epsilon}))|\leq\frac{1}{2}C(L+1)a=Va \\
  &|f(\mu^{j}_{t},\beta_{j}(\phi_{t}))-f(\mu^{j}_{t},\beta_{j}(\phi_{\ell\epsilon}))|\leq \frac{1}{2}C(L+1)a=Va.
\end{align*}
The above imply that
\begin{align*}
  \int_{\ell\epsilon}^{(\ell+1)\epsilon}f(\mu^{\epsilon,j}_{t},\beta_{j}(\phi_{t}^{\epsilon}))dt &\leq \int_{\ell\epsilon}^{(\ell+1)\epsilon}f(\mu^{\epsilon,j}_{t},\beta_{j}(\phi_{\ell\epsilon}))dt +\epsilon Va\\
   &= \epsilon f(\mu^{\epsilon,j}_{\ell\epsilon},\beta_{j}(\phi_{\ell\epsilon}))+\epsilon Va \\
   &\leq \int_{\ell\epsilon}^{(\ell+1)\epsilon}f(\mu^{j}_{t},\beta_{j}(\phi_{\ell\epsilon}))dt +\epsilon Va \\
   &\leq \int_{\ell\epsilon}^{(\ell+1)\epsilon}f(\mu^{j}_{t},\beta_{j}(\phi_{t}))dt +2Va\epsilon
\end{align*}
where the second inequality follows from Jensen's inequality. Therefore
\begin{equation*}
  I_{T}(\phi^{\epsilon}|\mu^{\epsilon})\leq I_{T}(\phi|\mu)+2VTa
\end{equation*}
We can now choose $a<\min\{\bar{a}_{\eta},\tilde{a}_{\eta}, \eta/2VT\}$ to have our result.
\end{proof}

The next lemma states a large deviation estimate for Poisson random variables.
\begin{lemma}\label{le17}
 Let $Y_{1}$,$Y_{2}$,...be independent Poisson random variables with mean $\sigma\epsilon$. For all $N\in\mathbb{N}$, let
\begin{equation*}
 \bar{Y}^{N}=\frac{1}{N}\sum_{n=0}^{N}Y_{n}.
\end{equation*}
For any $s>0$ there exist $K, \epsilon_{0}>0$ and $N_{0}\in\mathbb{N}$ such that taking $g(\epsilon)=K\sqrt{\log^{-1}(\epsilon^{-1})}$ we have
\begin{equation*}
 \mathbb{P}^{N}(\bar{Y}^{N}>g(\epsilon))<\exp\{-sN\}
\end{equation*}
for all $\epsilon<\epsilon_{0}$ and $N>N_{0}$.
\end{lemma}
\begin{proof}
 We apply the Gramer's theorem see e.g \cite{Dembo2009} (chapter 2) \begin{equation*}
 \limsup_{N\to\infty}\frac{1}{N}\log(\mathbb{P}^{N}(\bar{Y}^{N}>g(\epsilon)))\leq-\inf_{x\geq g(\epsilon)}\Lambda^{*}_{\epsilon}(x)
\end{equation*}
where $\Lambda^{*}_{\epsilon}(x)=\sup_{\lambda\in\mathbb{R}}\{\lambda x-\Lambda_{\epsilon}(\lambda)\}$ with 
\begin{align*}
 &\Lambda_{\epsilon}(\lambda)=\log(\mathbb{E}(e^{\lambda Y_{1}})=\sigma\epsilon(e^{\lambda}-1).
\end{align*}
We deduce that 
\begin{equation*}
 \Lambda^{*}_{\epsilon}(x)=x\log\frac{x}{\sigma\epsilon}-x+\sigma\epsilon.
\end{equation*}
This last function is convex It reaches its infimum at $x=\sigma\epsilon$ and as $\lim_{\epsilon\rightarrow0}\frac{g(\epsilon)}{\sigma\epsilon}=+\infty$
there exists $\epsilon_{1}>0$ such that $g(\epsilon)>\sigma\epsilon$ for all $\epsilon<\epsilon_{1}$ and then
\begin{align*}
 \inf_{x\geq g(\epsilon)}\Lambda^{*}_{\epsilon}(x)&=g(\epsilon)\log\frac{g(\epsilon)}{\sigma\epsilon}-g(\epsilon)+\sigma\epsilon\\
 &=g(\epsilon)\log(g(\epsilon))-g(\epsilon)\log(\sigma\epsilon)-g(\epsilon)+\sigma\epsilon\\
 &\approx K\sqrt{\log(1/\epsilon)}\rightarrow\infty\quad\text{as}\quad\epsilon\rightarrow0.
\end{align*}
Then there exists $\epsilon_{2}>0$ such that $\inf_{x\geq g(\epsilon)}\Lambda^{*}_{\epsilon}(x)>s$ for all $\epsilon<\epsilon_{2}$.
\par Taking $\epsilon_{0}=\min\{\epsilon_{1}, \epsilon_{2}\}$, we have the lemma.
\end{proof}

\section{The Lower Bound} \label{Seclower}

 We first prove that for $z\in A$, $\phi\in D_{T,A}$, $\phi_{0}=z$ and any $\eta>0$, $\delta>0$ there exist $\tilde{\delta}>0$ and $N_{\eta,\delta}$, such that for all $y$, $|y-z|<\tilde{\delta}$ and any $N>N_{\eta,\delta}$, we have
\begin{equation}\label{loexplower}
  \mathbb{P}_{y}(\|Z^{N}-\phi\|_{T}<\delta)\geq\exp\{-N(I_{T}(\phi)+\eta)\},
\end{equation}
Where $\xi_{T}$ is defined by \eqref{likelihood}.

To this end, it is enough to prove \eqref{loexplower} considering $\phi\in\mathcal{AC}_{T,A}$ because the inequality is true when $I_{T}(\phi)=\infty$.
We apply some lemmas of the preceding section to show that it is enough to consider some suitable paths $\phi$ with the $\mu\in\mathcal{A}_{d}(\phi)$.

The goal of the next lemma is to establish a crucial inequality to deduce \eqref{loexplower}.
\begin{lemma}\label{preliapprox}
For $z\in A$, $\phi\in\mathcal{AC}_{T,A}$, $\phi_{0}=z$, there exists $a_{0}$ such that for any $a<a_{0}$, $\epsilon>0$ the polygonal approximation $\phi^{\epsilon}$ of $\phi^{a}$  defined by 
\begin{equation}\label{approxa}
  \phi^{\epsilon}_{t}=\phi^{a}_{\ell\epsilon}\frac{(\ell+1)\epsilon-t}{\epsilon}+\phi^{a}_{(\ell+1)\epsilon}\frac{t-\ell\epsilon}{\epsilon}\quad \forall t\in[\ell\epsilon, (\ell+1)\epsilon[,
\end{equation}
satisfies the following assertion:

 For any $\mu^{\epsilon}\in\mathcal{A}_{d}(\phi^{\epsilon})$ constant over the time intervals $[\ell\epsilon, (\ell+1)\epsilon[$ and  bounded above by some constant $L>0$, any $\eta>0$ and suitable small $\delta>0$ there exist $0<\tilde{\delta}<\delta$ and $N_{\eta,\delta,\tilde{\delta}}\in\mathbb{N}$ such that for all $y$, $|y-z|<\tilde{\delta}$ and any $N>N_{\eta,\delta,\tilde{\delta}}$
\begin{equation*}
  \mathbb{P}_{y}(\|Z^{N}-\phi^{\epsilon}\|_{T}<\delta)\geq\exp\{-N(I_{T}(\phi^{\epsilon}|\mu^{\epsilon})+\eta)\}.
\end{equation*}
\end{lemma}
\begin{proof} 
Note that $\mu^{\epsilon}$ can be choose as in Lemma \ref{le13}. We define some events $B_{j}$, $j=1,...,k$ for controlling the likelihood ratio. For $\gamma>0$ let
\begin{equation*}
  B_{j}=\Big\{\Big|\sum_{p=1}^{Q}\delta_{p}(j)\log\Big(\frac{\beta_{j}(Z^{N}(\tau_{p}^{-}))}{\mu^{\epsilon,j}_{\lfloor\tau_{p}/\epsilon\rfloor\epsilon}}\Big) -N\sum_{\ell=1}^{T/\epsilon}\mu^{\epsilon,j}_{\ell\epsilon}\log\Big(\frac{\beta_{j}(\phi^{\epsilon}_{\ell\epsilon})}{\mu^{\epsilon,j}_{\ell\epsilon}}\Big)\epsilon\Big|\leq N\gamma\Big\}
\end{equation*}
Where $Q$ was introduced first above theorem \ref{girsanov}.

In what follows we put $\tilde{\beta}_{j}(Z^{N}(t))=\mu^{\epsilon,j}_{t}$ and we have on $\{\|Z^{N}-\phi^{\epsilon}\|_{T}<\delta\}\cap(\bigcap_{j=1}^{k}B_{j})=\{\|Z^{N}-\phi^{\epsilon}\|_{T}<\delta\}\cap B$,
\begin{align*}
  &\xi^{-1}_{T}=\exp\Big\{\sum_{p=1}^{Q}\sum_{j=1}^{k}\delta_{p}(j)\log\Big(\frac{\beta_{j}(Z^{N}(\tau_{p}^{-}))}{\mu^{\epsilon,j}_{\tau_{p}^{-}}}\Big) +N\int_{0}^{T}\sum_{j=1}^{k}(\mu^{\epsilon,j}_{t}-\beta_{j}(Z^{N}(t)))dt\Big\}\\
  &\geq\exp\Big\{-N\sum_{\ell=1}^{T/\epsilon}\sum_{j=1}^{k}\mu^{\epsilon,j}_{\ell\epsilon}\log\Big(\frac{\mu^{\epsilon,j}_{\ell\epsilon}}{\beta_{j}(\phi^{\epsilon}_{\ell\epsilon})}\Big)\epsilon +N\int_{0}^{T}\sum_{j=1}^{k}(\mu^{\epsilon,j}_{t}-\beta_{j}(Z^{N}(t)))dt-kN\gamma\Big\}\\
  &\geq\exp\Big\{-N\sum_{\ell=1}^{T/\epsilon}\sum_{j=1}^{k}\mu^{\epsilon,j}_{\ell\epsilon}\log\Big(\frac{\mu^{\epsilon,j}_{\ell\epsilon}}{\beta_{j}(\phi^{\epsilon}_{\ell\epsilon})}\Big)\epsilon +N\int_{0}^{T}\sum_{j=1}^{k}(\mu^{\epsilon,j}_{t}-\beta_{j}(\phi^{\epsilon}_{t}))dt-N(kTC\delta+k\gamma)\Big\}\\
\end{align*}
We note here that the first inequality is true because the $\mu^{\epsilon,j}_{t}$ is constant on the intervals $[\ell\epsilon, (\ell+1)\epsilon[$ and the second one come from the Lipschitz continuity of the rates $\beta_{j}$.
Since the integrand is continuous, we deduce from the convergence of the Riemann sums that when  $\epsilon$ is small enough we have 
\begin{align*}
  &\xi^{-1}_{T}\geq\exp\Big\{-N\int_{0}^{T}\sum_{j=1}^{k}\Big[\mu^{\epsilon,j}_{t}\log\Big(\frac{\mu^{\epsilon,j}_{t}}{\beta_{j}(\phi^{\epsilon}_{t})}\Big) -\mu^{\epsilon,j}_{t}+\beta_{j}(\phi^{\epsilon}_{t})\Big]dt-N (kTC\delta+k\gamma)\Big\}\\
   &\geq\exp\{-N(I_{T}(\phi^{\epsilon}|\mu^{\epsilon})+(kTC\delta+k\gamma))\}\quad\text{on the event}\quad\{\|Z^{N}-\phi^{\epsilon}\|_{T}<\delta\}\cap B.
\end{align*}
Then  for any $\eta>0$, there exists $\delta>0$ and $\gamma>0$ such that for $N$ large enough we have
\begin{equation*}
  \xi^{-1}_{T}\geq\exp\{-N(I_{T}(\phi^{\epsilon}|\mu^{\epsilon})+\eta/2)\}
\end{equation*}
Moreover from corollary \ref{cogirsanov}
\begin{align*}
  &\mathbb{P}_{y}(\|Z^{N}-\phi^{\epsilon}\|_{T}<\delta)\geq\widetilde{\mathbb{E}}\Big(\xi^{-1}_{T}.\mathfrak{1}_{\{\|Z^{N}-\phi^{\epsilon}\|_{T}<\delta\}}\Big)\\
  &\geq\widetilde{\mathbb{E}}_{y}\Big(\xi^{-1}_{T}.\mathfrak{1}_{\{\{\|Z^{N}-\phi^{\epsilon}\|_{T}<\delta\}\cap B\}}\Big)\\
  &\geq\exp\{-N(I_{T}(\phi^{\epsilon}|\mu^{\epsilon})+\eta/2)\}\widetilde{\mathbb{P}}_{y}(\{\|Z^{N}-\phi^{\epsilon}\|_{T}<\delta\}\cap B)
\end{align*}
\end{proof}

To conclude this proof it is enough to establish the following lemma:

\begin{lemma}
For $z\in A$, $\phi\in\mathcal{AC}_{T,A}$, $\phi_{0}=z$, there exists $a_{0}$ such that for any $a<a_{0}$, $\epsilon>0$ the polygonal approximation $\phi^{\epsilon}$ of $\phi^{a}$ defined by \eqref{approxa} has the property that there exists $\tilde{\delta}>0$ such that for all $y$, $|y-z|<\tilde{\delta}$
\begin{equation*}
  \lim_{N\to\infty}\widetilde{\mathbb{P}}_{y}(\{\|Z^{N}-\phi^{\epsilon}\|_{T}<\delta\}\cap B)=1
\end{equation*}
\end{lemma}
\begin{proof}
It is enough to prove that $\lim_{N\to\infty}\widetilde{\mathbb{P}}_{y}(\|Z^{N}-\phi^{\epsilon}\|_{T}<\delta)=1$ and that for all $1\leq j\leq k$, $\lim_{N\to\infty}\widetilde{\mathbb{P}}_{y}(\{\|Z^{N}-\phi^{\epsilon}\|_{T}<\delta\}\cap B_{j}^{c})=0$. The first limit follows from Theorem \ref{LLN} for processes under the probability $\widetilde{\mathbb{P}}_{y}$ provided that we choose $a_{0}$ and $\tilde{\delta}<\delta/2$ in suitable way.
We now establish that $\widetilde{\mathbb{P}}_{y}(\|Z^{N}-\phi^{\epsilon}\|_{T}<\delta\cap B_{j}^{c})\to0$ as $N\to\infty$, for any $1\leq j\leq k$.

We have $\sup_{p}|Z^{N}(\tau_{p})-\phi^{\epsilon}_{\tau_{p}}|<\delta$ on $\{\|Z^{N}-\phi^{\epsilon}\|_{T}<\delta\}$ and we can choose $\epsilon$ small enough such that $\sup_{p}|\phi^{\epsilon}_{\tau_{p}}-\phi^{\epsilon}_{\lfloor\tau_{p}/\epsilon\rfloor\epsilon}|<\delta$ and thus $\sup_{p}|Z^{N}(\tau_{p})-\phi^{\epsilon}_{\lfloor\tau_{p}/\epsilon\rfloor\epsilon}|<2\delta$.

Note that we have on $\{\|Z^{N}-\phi^{\epsilon}\|_{T}<\delta\}$
\begin{align*}
  \Big|\sum_{p=1}^{Q}\delta_{p}(j)\log\Big(\frac{\beta_{j}(Z^{N}(\tau_{p}^{-}))}{\mu^{\epsilon,j}_{\lfloor\tau_{p}/\epsilon\rfloor\epsilon}}\Big) -\sum_{p=1}^{Q}\delta_{p}(j)\log\Big(\frac{\beta_{j}(\phi^{\epsilon}_{\lfloor\tau_{p}/\epsilon\rfloor\epsilon})}{\mu^{\epsilon,j}_{\lfloor\tau_{p}/\epsilon\rfloor\epsilon}}\Big)\Big|
  &\leq\Big|\sum_{p=1}^{Q}\delta_{p}(j)\log\Big(\frac{\beta_{j}(Z^{N}(\tau_{p}^{-}))}{\beta_{j}(\phi^{\epsilon}_{\lfloor\tau_{p}/\epsilon\rfloor\epsilon})}\Big) \Big| \\ 
  &\leq\frac{2CQ\delta}{C_{a}}\\
\end{align*}
since $|\beta_{j}(Z^{N}(\tau_{p}^{-}))-\beta_{j}(\phi^{\epsilon}_{\lfloor\tau_{p}/\epsilon\rfloor\epsilon})|<2C\delta$. Let $m_{\ell}$ the number of jumps in the interval $[(\ell-1)\epsilon,\ell\epsilon[$ we have
\begin{align*}
  &\Big|\sum_{p=1}^{Q}\delta_{p}(j)\log\Big(\frac{\beta_{j}(Z^{N}(\tau_{p}^{-}))}{\mu^{\epsilon,j}_{\lfloor\tau_{p}/\epsilon\rfloor\epsilon}}\Big) -N\sum_{\ell=1}^{T/\epsilon}\mu^{\epsilon,j}_{\ell\epsilon}\log\Big(\frac{\beta_{j}(\phi^{\epsilon}_{\ell\epsilon})}{\mu^{\epsilon,j}_{\ell\epsilon}}\Big)\epsilon\Big| \\
  &\leq\Big|\sum_{p=1}^{Q}\delta_{p}(j)\log\Big(\frac{\beta_{j}(\phi^{\epsilon}_{\lfloor\tau_{p}/\epsilon\rfloor\epsilon})}{\mu^{\epsilon,j}_{\lfloor\tau_{p}/\epsilon\rfloor\epsilon}}\Big)-N\sum_{\ell=1}^{T/\epsilon}\mu^{\epsilon,j}_{\ell\epsilon}\log\Big(\frac{\beta_{j}(\phi^{\epsilon}_{\ell\epsilon})}{\mu^{\epsilon,j}_{\ell\epsilon}}\Big)\epsilon \Big| \\
  &+\Big|\sum_{p=1}^{Q}\delta_{p}(j)\log\Big(\frac{\beta_{j}(Z^{N}(\tau_{p}^{-}))}{\mu^{\epsilon,j}_{\lfloor\tau_{p}/\epsilon\rfloor\epsilon}}\Big) -\sum_{p=1}^{Q}\delta_{p}(j)\log\Big(\frac{\beta_{j}(\phi^{\epsilon}_{\lfloor\tau_{p}/\epsilon\rfloor\epsilon})}{\mu^{\epsilon,j}_{\lfloor\tau_{p}/\epsilon\rfloor\epsilon}}\Big) \Big| \\
  &\leq\Big|\sum_{\ell=1}^{T/\epsilon}\log\Big(\frac{\beta_{j}(\phi^{\epsilon}_{\ell\epsilon})}{\mu^{\epsilon,j}_{\ell\epsilon}}\Big) \Big(\sum_{p=1}^{m_{\ell}}\delta_{p}(j)-N\mu^{\epsilon,j}_{\ell\epsilon}\epsilon\Big)\Big|+\frac{2CQ\delta}{C_{a}}.
 \end{align*}
As the rate of jumps are constant on the interval $[(\ell-1)\epsilon,\ell\epsilon[$ under $\widetilde{\mathbb{P}}^{N}$, $\sum_{p=1}^{m_{\ell}}\delta_{p}(j)$ is the number of jumps of a Poisson process $P_{j}$ on this interval. So it is a Poisson random variable with mean $N\mu^{\epsilon,j}_{\ell\epsilon}\epsilon$. We deduce from Chebyshev's inequality that
\begin{equation*}
  \widetilde{\mathbb{P}}_{y}\Big(\Big|\log\Big(\frac{\beta_{j}(\phi^{\epsilon}_{\ell\epsilon})}{\mu^{\epsilon,j}_{\ell\epsilon}}\Big) \Big(\sum_{p=1}^{m_{\ell}}\delta_{p}(j)-N\mu^{\epsilon,j}_{\ell\epsilon}\epsilon\Big)\Big|>\frac{N\gamma\epsilon}{2T}\Big) \leq\frac{4T^{2}\sup_{\ell\leq T/\epsilon}\Big(\log^{2}\Big(\frac{\beta_{j}(\phi^{\epsilon}_{\ell\epsilon})}{\mu^{\epsilon,j}_{\ell\epsilon}}\Big)N\mu^{\epsilon,j}_{\ell\epsilon}\epsilon\Big)}{N^{2}\gamma^{2}\epsilon^{2}}.
\end{equation*}
As $C_{a}\leq\beta_{j}(\phi^{\epsilon}_{t})\leq \sigma$ and $\mu^{\epsilon,j}_{t}\leq L$ we have $\sup_{\ell\leq T/\epsilon}\Big(\log^{2}\Big(\frac{\beta_{j}(\phi^{\epsilon}_{\ell\epsilon})}{\mu^{\epsilon,j}_{\ell\epsilon}}\Big)\mu^{\epsilon,j}_{\ell\epsilon}\Big)\leq C(L,a)$. Thus
\begin{align*}
  \widetilde{\mathbb{P}}_{y}(\|Z^{N}-\phi^{\epsilon}\|_{T}<\delta\}\cap B_{j}^{c})&\leq \widetilde{\mathbb{P}}_{y}\Big(\Big|\sum_{\ell=1}^{T/\epsilon}\log\Big(\frac{\beta_{j}(\phi^{\epsilon}_{\ell\epsilon})}{\mu^{\epsilon,j}_{\ell\epsilon}}\Big) \Big(\sum_{p=1}^{m_{\ell}}\delta_{p}(j)-N\mu^{\epsilon,j}_{\ell\epsilon}\epsilon\Big)\Big|+\frac{2CQ\delta}{C_{a}}>N\gamma\Big) \\
  &\leq\frac{4T^{2}C(L,a)}{N\gamma^{2}\epsilon}+\widetilde{\mathbb{P}}_{y}\Big(\frac{2CQ\delta}{C_{a}}\geq \frac{N\gamma}{2}\Big).
\end{align*}
The number of jumps during the period time $T$ under the probability $\widetilde{\mathbb{P}}_{z}$ is the sum of $T/\epsilon$ Poisson random variables with mean $N\sum_{j=1}^{k}\mu^{\epsilon,j}_{\ell\epsilon}\epsilon$. we take $\gamma=\frac{8C\delta}{C_{a}}\sum_{\ell=1}^{T/\epsilon}\sum_{j=1}^{k}\mu^{\epsilon,j}_{\ell\epsilon}\epsilon$  where $\delta$ is chosen such that $\delta/C_{a}$ is small. Therefore, as long as $\sum_{\ell=1}^{T/\epsilon}\sum_{j=1}^{k}\mu^{\epsilon,j}_{\ell\epsilon}>0$, the law of large number for Poisson variables give us
\begin{equation*}
  \widetilde{\mathbb{P}}_{y}\Big(\frac{2CQ\delta}{C_{a}}\geq \frac{N\gamma}{2}\Big)=\widetilde{\mathbb{P}}_{y}\Big(\frac{Q}{N}\geq 2\sum_{\ell=1}^{T/\epsilon}\sum_{j=1}^{k}\mu^{\epsilon,j}_{\ell\epsilon}\epsilon\Big)\rightarrow0
\end{equation*}
as $N\rightarrow\infty$.
\end{proof}

We now deduce from Lemma \ref{preliapprox} the next result follows the argument from
in the proof of  Lemma 3 in \cite{dolgoarshinnykhsample}.
\begin{lemma}\label{lower1}
For $z\in A$, $\phi\in\mathcal{AC}_{T,A}$, $\phi_{0}=z$ and any $\eta>0$, $\delta>0$ there exist $\tilde{\delta}>0$ and $N_{\eta,\delta}$ such that for all $N>N_{\eta,\delta}$,
\begin{equation}
  \inf_{y:|y-z|<\tilde{\delta}}\mathbb{P}_{y}(\|Z^{N}-\phi\|_{T}<\delta)\geq\exp\{-N(I_{T}(\phi)+\eta)\}.
\end{equation}
\end{lemma}
\begin{proof}
For $\delta, \eta>0$ let $\phi\in\mathcal{AC}_{T,A}$, $\phi_{0}=z$ such that $I_{T}(\phi)<\infty$ then using Lemma \ref{le11} we have that there exists  $a_{\eta}>0$ such that for all $a<a_{\eta}$ there exists $\phi^{a}\in R^{a}$ such that $\|\phi-\phi^{a}\|_{T}<c_{1}a$ and $I_{T}(\phi^{a})\leq I_{T}(\phi)+\eta/4$. 
As $I_{T}(\phi^{a})<\infty$ using the lemma \ref{le12} we deduce that there exists $L>0$ and $\phi^{a,L}\in R^{a/2}$ is such that $\|\phi^{a}-\phi^{a,L}\|_{T}<c_{1}\frac{a}{2}$ and $I_{T}(\phi^{a,L}|\mu^{a,L})\leq I_{T}(\phi^{a})+\eta/4$ where $\mu^{a,L}\in\mathcal{A}_{d}(\phi^{a,L})$ such that $\mu^{a,L,j}_{t}<L$, $j=1,...,k$. 
Now we can deduce from Lemma \ref{le13} that for all $\epsilon>0$ the polygonal approximation $\phi^{a,L,\epsilon}$ of $\phi^{a,L}$ satisfies $\|\phi^{a,L}-\phi^{a,L,\epsilon}\|_{T}<c_{1}\frac{a}{2}$ and $I_{T}(\phi^{a,L,\epsilon}|\mu^{a,L,\epsilon})\leq I_{T}(\phi^{a,L}|\mu^{a,L})+\eta/4$ where $\mu^{a,L,\epsilon}\in\mathcal{A}_{d}(\phi^{a,L,\epsilon})$ is such that $\mu^{a,L,\epsilon,j}_{t}<L$, $j=1,...,k$.
Now we choose $a$ such that $2c_{1}a<\delta/2$ and we have 
\begin{align*}
  \inf_{y:|y-z|<\tilde{\delta}}\mathbb{P}_{y}\Big(\|Z^{N}-\phi\|_{T}<\delta\Big)&\geq \inf_{y:|y-z|<\tilde{\delta}}\mathbb{P}_{y}\Big(\|Z^{N}-\phi\|_{T}<\frac{\delta}{2}+2c_{1}a\Big) \\
  &\geq \inf_{y:|y-z|<\tilde{\delta}}\mathbb{P}_{y}\Big(\|Z^{N}-\phi^{a}\|_{T}<\frac{\delta}{2}+c_{1}a\Big) \\
  &\geq\inf_{y:|y-z|<\tilde{\delta}}\mathbb{P}_{y}\Big(\|Z^{N}-\phi^{a,L}\|_{T}<\frac{\delta}{2}+c_{1}\frac{a}{2}\Big) \\
  &\geq\inf_{y:|y-z|<\tilde{\delta}}\mathbb{P}_{y}\Big(\|Z^{N}-\phi^{a,L,\epsilon}\|_{T}<\frac{\delta}{2}\Big)\\
  &\geq\exp\{-N(I_{T}(\phi^{a,L,\epsilon}|\mu^{a,L,\epsilon})+\eta/4)\}\\
  &\geq\exp\{-N(I_{T}(\phi^{a,L}|\mu^{a,L})+\eta/2)\}\\
  &\geq\exp\{-N(I_{T}(\phi^{a})+3\eta/4)\}\\
  &\geq\exp\{-N(I_{T}(\phi)+\eta)\}
\end{align*}
where we have used the lemma \ref{preliapprox} at the $5^{th}$ inequality.
\end{proof}

 We finish the proof of the lower bound by the following theorem
\begin{theorem}\label{lowerbound1}
 For any  open subset $G$ of $D_{T,A}$ and $z\in A$,
\begin{equation}\label{lower}
\liminf_{N\to\infty}\frac{1}{N}\log\mathbb{P}_{z}^{N}(G)\geq -\inf_{\phi\in G,\phi_{0}=z}I_{T}(\phi).
\end{equation}
\end{theorem}

\begin{proof}
Note that in fact \eqref{loexplower} and \eqref{lower} are equivalent. We only have to show that \eqref{lower} follows from \eqref{loexplower}. To this end let $I=\inf_{\phi\in G,\phi_{0}=z}I_{T}(\phi)<\infty$ then, for $\eta>0$ there exists a $\phi^{\eta}\in G$, $\phi^{\eta}_{0}=z$ such that $I_{T}(\phi^{\eta})\leq I+\eta$. Moreover we can choose $\delta=\delta(\phi^{\eta})$ small enough such that $\{\|Z^{N}-\phi^{\eta}\|_{T}<\delta\}\subset G$. And then $\mathbb{P}_{z}(\|Z^{N}-\phi^{\eta}\|_{T}<\delta)\leq \mathbb{P}_{z}^{N}(G)$. This implies from the inequality \eqref{loexplower} that for all $\eta>0$,
\begin{align*}
\liminf_{N\to\infty}\frac{1}{N}\log\mathbb{P}_{z}^{N}(G)&\geq\liminf_{N\to\infty}\frac{1}{N}\log\mathbb{P}_{z}(\|Z^{N}-\phi^{\eta}\|_{T}<\delta) \\
&\geq-I_{T}(\phi^{\eta}) \\
&\geq -I-\eta
\end{align*}
and then 
\begin{equation*}
\liminf_{N\to\infty}\frac{1}{N}\log\mathbb{P}_{z}^{N}(G)\leq-I.
\end{equation*}
\end{proof}

\begin{corollary}
For any open subset $G$ of $D_{T,A}$ and any compact subset $K$ of $A$,
  \begin{equation*}
\liminf_{N\to\infty}\frac{1}{N}\log\inf_{z\in K}\mathbb{P}_{z}(Z^{N}\in G)\geq-\sup_{z\in K}\inf_{\phi\in G, \phi_{0}=z}I_{T}(\phi).
\end{equation*}
\end{corollary}

\begin{proof}
The arguments are the same as in the proof of Corollary 5.6.15 in \cite{Dembo2009}. Let
\begin{equation*}
I_{K}:=\sup_{z\in K}\inf_{\phi\in G, \phi_{0}=z}I_{T}(\phi).
\end{equation*}
For  $\eta>0$ fix, let $I_{K}^{\eta}:=\max\{I_{K}+\eta, \eta^{-1}\}$. Then from \eqref{lower1} it follows that for any $z\in K$, there exists a $N_{z}$ such that for all $N>N_{z}$ and $y\in B(z,\frac{1}{N_{z}})$,
\begin{equation*}
\frac{1}{N}\log\mathbb{P}_{y}(Z^{N}\in G)\geq-\inf_{\phi\in G, \phi_{0}=z}I_{T}(\phi)\geq-I_{K}^{\eta}.
\end{equation*}
And then
\begin{equation*}
\frac{1}{N}\log\inf_{y\in B(z,\frac{1}{N_{z}})}\mathbb{P}_{y}(Z^{N}\in G)\geq-I_{K}^{\eta}.
\end{equation*}
As $K$ is compact, there exits a finite sequence $(z_{i})_{1\leq i\leq m}\subset K$  such that  $K\subset\bigcup_{i=1}^{m}B(z_{i},\frac{1}{N_{z_{i}}})$. Then for $N\geq \max_{1\leq i\leq m}N_{z_{i}}$,
\begin{equation*}
\frac{1}{N}\log\inf_{y\in K}\mathbb{P}_{y}(Z^{N}\in G)\geq-I_{K}^{\eta}.
\end{equation*}
It first remains to take $\liminf$ as $N\to\infty$ and then let $\eta$ tend to $0$ to have result.
\end{proof}

\section{The Upper Bound} \label{Secupper}

For all $\phi\in D_{T,A}$,  and $F\subset D_{T,A}$ we define
\begin{equation}\label{hset}
  \rho_{T}(\phi, F)=\inf_{\psi\in F}\|\phi-\psi\|_{T}.
\end{equation}
For $z\in A$, $\delta,s>0$ we define the set
\begin{equation*}
   F_{\delta}^{s}=\{\phi\in D_{T,A}: \rho_{T}(\phi, \Phi(s))\geq\delta\},
\end{equation*} 
 where $\Phi(s)=\{\psi\in D_{T,A}: I_{T}(\psi)\leq s\}$.
 
We start by proving the following lemma which will be enough to conclude the upper bound. 
\begin{lemma}\label{upperb}
For $z\in A$, $\delta$, $\eta$ and $s>0$ there exists $N_{0}\in\mathds{N}$ such that
\begin{equation}\label{loexpupper}
  \mathbb{P}^{N}_{z}(F_{\delta}^{s})\leq\exp\{-N(s-\eta)\}
\end{equation}
whenever $N\geq N_{0}$.
\end{lemma}
\begin{proof}
Let $Z^{N}_{a}(t)=(1-a)Z^{N}(t)+az_{0}$ then $\|Z^{N}-Z^{N}_{a}\|<c_{1}a$ and for all $c_{1}a<\delta(d-1)/d$ we have
\begin{align*}
 \mathbb{P}^{N}_{z}(F_{\delta}^{s})&= \mathbb{P}_{z}\Big(\rho_{T}(Z^{N},\Phi(s))\geq\delta\Big) \\
 &\leq\mathbb{P}_{z}\Big(\rho_{T}(Z^{N}_{a},\Phi(s))\geq\frac{\delta}{d}\Big).
\end{align*}
We now approximate the paths $Z^{N}$ by smoother paths. Let $\epsilon>0$ be such that $T/\epsilon\in\mathds{N}$. We construct a polygonal approximation of $Z^{N}_{a}$ defined for all $t\in [\ell\epsilon,(\ell+1)\epsilon[$ by
\begin{equation*}
  \Upsilon_{t}=\Upsilon^{a,\epsilon}_{t}=Z^{N}_{a}(\ell\epsilon)\frac{(\ell+1)\epsilon-t}{\epsilon}+Z^{N}_{a}((\ell+1)\epsilon)\frac{t-\ell\epsilon}{\epsilon}.
\end{equation*}
The event $\{\|Z^{N}_{a}-\Upsilon\|_{T}<\frac{\delta}{2d}\}\cap\{\rho_{T}(Z^{N}_{a},\Phi(s))\geq\frac{\delta}{d}\}$ is contained in $\{\rho_{T}(\Upsilon,\Phi(s))\geq\frac{\delta}{2d}\}$ and
\begin{align}\label{major2}
  \mathbb{P}_{z}\Big(\rho_{T}(Z^{N}_{a},\Phi(s))\geq\frac{\delta}{d}\Big)&\leq\mathbb{P}_{z}\Big(\rho_{T}(\Upsilon,\Phi(s))\geq\frac{\delta}{2d}\Big) +\mathbb{P}_{z}\Big(\{\|Z^{N}_{a}-\Upsilon\|_{T}\geq\frac{\delta}{2d}\}\Big) \nonumber\\
  &\leq\mathbb{P}_{z}(I_{T}(\Upsilon)\geq s)+\mathbb{P}_{z}\Big(\|Z^{N}_{a}-\Upsilon\|_{T}\geq\frac{\delta}{2d}\Big)
\end{align}
We now  bound $\mathbb{P}_{z}(I_{T}(\Upsilon)\geq s)$. For any choice $\mu\in\mathcal{A}_{d}(\Upsilon)$ we have $I_{T}(\Upsilon)\leq I_{T}(\Upsilon|\mu)$ and
\begin{equation*}
  \mathbb{P}_{z}(I_{T}(\Upsilon)\geq s)\leq\mathbb{P}_{z}(I_{T}(\Upsilon|\mu)\geq s).
\end{equation*}
Let $\mu^{j}_{t}$, $j=1,...,k$ be constant on the intervals $[\ell\epsilon,(\ell+1)\epsilon[$ and equal to
\begin{equation}\label{mu1}
  \mu^{j}_{t}=\frac{1-a}{N\epsilon}\Big[P_{j}\Big(N\int_{0}^{(\ell+1)\epsilon}\beta_{j}(Z^{N}(s)ds\Big)-P_{j}\Big(N\int_{0}^{\ell\epsilon}\beta_{j}(Z^{N}(s)ds\Big)\Big]
\end{equation}
Since $\Upsilon$ is piecewise linear, for $t\in]\ell\epsilon,(\ell+1)\epsilon[$
\begin{equation*}
  \frac{d\Upsilon_{t}}{dt}=\frac{(1-a)}{\epsilon}(Z^{N}((\ell+1)\epsilon)-Z^{N}(\ell\epsilon))=\sum_{j=1}^{k}\mu^{j}_{t}h_{j}.
\end{equation*}
Then the $\mu^{j}_{t}$ given by \eqref{mu1} belong to $\mathcal{A}_{d}(\Upsilon)$.
\par To control the change in $\Upsilon$ over the intervals of length $\epsilon$ define  $g(\epsilon)=K\sqrt{\log^{-1}(\epsilon^{-1})}$ where $K>0$ is fixed, and define a collection of events $B=\{B_{\epsilon}\}_{\epsilon>0}$
\begin{equation*}
 B_{\epsilon}=\bigcap_{\ell=0}^{T/\epsilon-1}B_{\epsilon}^{\ell}
\end{equation*}
where
\begin{equation*}
  B_{\epsilon}^{\ell}=\Big\{\sup_{\ell\epsilon\leq t_{1},t_{2}\leq (\ell+1)\epsilon}|Z^{N}_{i}(t_{1})-Z^{N}_{i}(t_{2})|\leq g(\epsilon)\quad\text{for}\quad i=1,...,d\Big\}.
\end{equation*}
We have
\begin{equation}\label{major1}
  \mathbb{P}_{z}(I_{T}(\Upsilon|\mu)>s)\leq\mathbb{P}_{z}(\{I_{T}(\Upsilon|\mu)>s\}\cap B_{\epsilon})+\mathbb{P}(B^{c}_{\epsilon})
\end{equation}
and using the Chebyshev inequality we have that for all $0<\alpha<1$
\begin{equation}\label{cheby}
  \mathbb{P}_{z}(\{I_{T}(\Upsilon|\mu)>s\}\cap B_{\epsilon})\leq\frac{\mathbb{E}_{z}(\exp\{\alpha N I_{T}(\Upsilon|\mu)\}\mathfrak{1}_{B_{\epsilon}})}{\exp\{\alpha N s\}}.
\end{equation}
We need to show that the expectation above is appropriately small for $\alpha$ arbitrarily close to 1. For this we first prove the following lemma

\begin{lemma}
For all $0<\alpha<1$, $j=1,...,k$ and $\ell=0,...,T/\epsilon-1$, there exist $Z^{-}_{j}$ and $Z^{+}_{j}$ which conditionally upon $\mathcal{F}_{\ell}$ are Poisson random variables with mean  $N\epsilon\beta^{j-}_{\ell}=N\epsilon(\beta_{j}(Z^{N}(\ell\epsilon))-C d g(\epsilon))_{+}$ and $N\epsilon\beta^{j+}_{\ell}=N\epsilon(\beta_{j}(Z^{N}(\ell\epsilon))+C d g(\epsilon))$ respectively such that if
\begin{equation*}
 \Theta_{j}^{\ell}=\exp\Big\{\alpha N \int_{\ell\epsilon}^{(\ell+1)\epsilon}f(\mu^{j}_{t}, \beta_{j}(\Upsilon_{t}))dt\Big\}\mathfrak{1}_{B_{\epsilon}^{\ell}}
\end{equation*}
and
\begin{align*}
\Xi_{j}^{\ell}&=\exp\{2\alpha NCdg(\epsilon)\epsilon\}\times\Big[\exp\Big\{\alpha N\epsilon f\Big(\frac{(1-a)Z^{-}_{j}}{\epsilon N}, \beta^{a,j}_{\ell}\Big)\Big\}\\
&+\exp\Big\{\alpha N\epsilon f\Big(\frac{(1-a)Z^{+}_{j}}{\epsilon N}, \beta^{a,j}_{\ell}\Big)\Big\}\Big]
\end{align*}
with $\beta^{a,j}_{\ell}=(\beta_{j}(\Upsilon_{\ell\epsilon})-C d g(\epsilon))_{+}$, then

\begin{equation}\label{convex1}
  \Theta_{j}^{\ell}\leq\Xi_{j}^{\ell}\quad\text{a.s}
\end{equation}
\end{lemma}

\begin{proof}
On $B^{\ell}_{\epsilon}$, with $\epsilon$ such that $g(\epsilon)<1$ and $t\in[\ell\epsilon, (\ell+1)\epsilon]$,
using the Lipshitz continuity of the rates $\beta_{j}$ we have
\begin{equation*}
  |\beta_{j}(Z^{N}(t))-\beta_{j}(Z^{N}(\ell\epsilon))|\leq C|Z^{N}(t)-Z^{N}(\ell\epsilon)|\leq C d g(\epsilon),\quad j=1,...,k
\end{equation*}
Then we have
\begin{equation*}
  \Big|N\int_{\ell\epsilon}^{(\ell+1)\epsilon}\beta_{j}(Z^{N}(t))dt-N\epsilon\beta_{j}(Z^{N}(\ell\epsilon))\Big|\leq N\epsilon C d g(\epsilon),\quad j=1,...,k.
\end{equation*}
As $\mu^{j}_{t}$, $j=1,...,k$ satisfy \eqref{mu1}, we can write
\begin{equation}\label{twoz1z2}
  \frac{(1-a)Z^{-}_{j}}{\epsilon N}\leq\mu^{j}_{\ell\epsilon}\leq\frac{(1-a)Z^{+}_{j}}{\epsilon N}\quad\text{a.s}.
\end{equation}
where for example
\begin{align*}
  Z^{-}_{j} &= P_{j}\Big(N\int_{0}^{\ell\epsilon}\beta_{j}(Z^{N}(s))ds+\epsilon N(\beta_{j}(Z^{N}(\ell\epsilon))-C d g(\epsilon))_{+}\Big)-P_{j}\Big(N\int_{0}^{\ell\epsilon}\beta_{j}(Z^{N}(s))ds\Big) \\
  Z^{+}_{j} &= P_{j}\Big(N\int_{0}^{\ell\epsilon}\beta_{j}(Z^{N}(s))ds+\epsilon N(\beta_{j}(Z^{N}(\ell\epsilon))+C d g(\epsilon))\Big)-P_{j}\Big(N\int_{0}^{\ell\epsilon}\beta_{j}(Z^{N}(s))ds\Big).
\end{align*}
Moreover it is easy to see that on $B^{\ell}_{\epsilon}$ we have
\begin{equation*}
  \max_{1\leq i\leq d}|\Upsilon_{t}^{i}-\Upsilon^{i}_{\ell\epsilon}|<(1-a)g(\epsilon)<g(\epsilon)\quad\text{for}\quad t\in[\ell\epsilon, (\ell+1)\epsilon].
\end{equation*}
And then
\begin{equation*}
  |\beta_{j}(\Upsilon_{t})-\beta_{j}(\Upsilon_{\ell\epsilon})|\leq C|\Upsilon_{t}-\Upsilon_{\ell\epsilon}|\leq C d g(\epsilon)
\end{equation*}
we deduce that
\begin{equation*}
  \beta_{j}(\Upsilon_{t})\geq(\beta_{j}(\Upsilon_{\ell\epsilon})-C d g(\epsilon))_{+}=\beta^{a,j}_{\ell}
\end{equation*}
and
\begin{equation*}
  \beta_{j}(\Upsilon_{t})\leq\beta_{j}(\Upsilon_{\ell\epsilon})+C d g(\epsilon)=\beta^{a,j}_{\ell}+2Cdg(\epsilon).
\end{equation*}
Thus
\begin{align*}
 & f(\mu^{j}_{t}, \beta_{j}(\Upsilon_{t}))=\mu_{t}^{j}\log\frac{\mu_{t}^{j}}{\beta_{j}(\Upsilon_{t})}-\mu_{t}^{j}+\beta_{j}(\Upsilon_{t})\\
&\leq\mu_{t}^{j}\log\frac{\mu_{t}^{j}}{\beta_{\ell}^{a,j}}-\mu_{t}^{j}+\beta_{\ell}^{a,j}+2Cdg(\epsilon)+\mu_{t}^{j}\log\frac{\beta_{\ell}^{a,j}}{\beta_{j}(\Upsilon_{t})}\\
&\leq f(\mu^{j}_{t}, \beta_{\ell}^{a,j})+2C d g(\epsilon)\quad\text{since}\quad\log\frac{\beta_{\ell}^{a,j}}{\beta_{j}(\Upsilon_{t})}<0.
\end{align*}
As $\mu^{j}_{t}=\mu^{j}_{\ell\epsilon}$ is constant over the interval $[\ell\epsilon, (\ell+1)\epsilon[$, we deduce that on $B^{\ell}_{\epsilon}$
\begin{equation}\label{ineqexpf}
  \exp\Big\{\alpha N \int_{\ell\epsilon}^{(\ell+1)\epsilon}f(\mu^{j}_{t}, \beta_{j}(\Upsilon_{t}))dt\Big\} \leq\exp\{\alpha N \epsilon f(\mu^{j}_{\ell\epsilon},\beta_{\ell}^{a,j})+2\alpha N C d \epsilon g(\epsilon)\}.
\end{equation}
From \eqref{twoz1z2}, \eqref{ineqexpf} and the convexity of $f(\nu,\omega)$ in $\nu$ we deduce the inequality of lemma.
\end{proof}

\par The next proposition gives us a bound for the conditionnal expectation of the right hand side of the inequality \eqref{convex1}.
\begin{proposition}
Let $a=h(\epsilon)=\Big[-\log g^{1/2}(\epsilon)\Big]^{-\frac{1}{\nu}}$. For all $0<\alpha<1$ there exist $\epsilon_{\alpha}$, $K_{\alpha}$ and $\tilde{K}$ such that for all $\epsilon\leq\epsilon_{\alpha}$ we have
\begin{align*}
  &\max_{q=-,+}\Big\{\mathbb{E}_{z}\Big(\exp\Big\{\alpha N\epsilon f\Big(\frac{(1-a)Z^{q}_{j}}{\epsilon N}, \beta^{a,j}_{\ell}\Big)\Big\}|\mathcal{F}^{N}_{\ell\epsilon}\Big)\Big\} \\
  &\leq K_{\alpha}\exp\{N\epsilon\tilde{K}(1-\alpha+2h(\epsilon)+2dg(\epsilon))\}.
\end{align*}
\end{proposition}

\begin{proof}
 Conditionally upon $\mathcal{F}^{N}_{\ell\epsilon}$, $Z_{j}^{q}$ is a Poisson variable with mean $N\epsilon\beta_{\ell}^{j,q}$. Moreover we have by the definition
\begin{equation*}
  \max\{|\beta_{\ell}^{a,j}-\beta_{\ell}^{j-}|, |\beta_{\ell}^{a,j}-\beta_{\ell}^{j+}|\}\leq \tilde{C}(a+2d g(\epsilon))
\end{equation*}
let $\tilde{\epsilon}=\epsilon/(1-a)$ and $\tilde{\alpha}=(1-a)\alpha$ then we have
\begin{align}\label{Inequal1}
 & \mathbb{E}_{z}\Big(\exp\Big\{\alpha N\epsilon f\Big(\frac{(1-a)Z^{q}_{j}}{\epsilon N}, \beta^{a,j}_{\ell}\Big)\Big\}|\mathcal{F}^{N}_{\ell\epsilon}\Big)=\mathbb{E}_{z}\Big(\exp\Big\{\alpha N\epsilon f\Big(\frac{Z^{q}_{j}}{\tilde{\epsilon} N}, \beta^{a,j}_{\ell}\Big)\Big\}|\mathcal{F}^{N}_{\ell\epsilon}\Big) \nonumber\\
  &=\sum_{m\geq0}\exp\Big\{\alpha N\epsilon f\Big(\frac{m}{\tilde{\epsilon} N}, \beta^{a,j}_{\ell}\Big)\Big\}\frac{(N\epsilon\beta_{\ell}^{j,q})^{m}\exp\{-N\epsilon\beta_{\ell}^{j,q}\}}{m!} \nonumber\\
  &=\sum_{m\geq0}\exp\Big\{\alpha N\epsilon \Big(\frac{m}{\tilde{\epsilon} N}\log\Big(\frac{m}{\tilde{\epsilon} N\beta^{a,j}_{\ell}}\Big)-\frac{m}{\tilde{\epsilon} N} +\beta^{a,j}_{\ell} \Big)\Big\}\frac{(N\epsilon\beta_{\ell}^{j,q})^{m}\exp\{-N\epsilon\beta_{\ell}^{j,q}\}}{m!} \nonumber\\
  &\leq\exp\{N\epsilon\tilde{C}(a+2d g(\epsilon))\}\sum_{m\geq0}\frac{m^{\tilde{\alpha}m}\exp\{-\tilde{\alpha}m\}}{m!} (N\epsilon\beta^{a,j}_{\ell})^{m(1-\tilde{\alpha})}\Big(\frac{\beta^{j,q}_{\ell}}{\beta^{a,j}_{\ell}}\Big)^{m} \exp\{-N\epsilon\beta^{a,j}_{\ell}(1-\alpha)\} \nonumber\\
  &\leq\exp\{N\epsilon C_{1}(a+2d g(\epsilon))\}\sum_{m\geq0}\frac{m^{\tilde{\alpha}m}\exp\{-\tilde{\alpha}m\}}{m!} (N\epsilon\beta^{a,j}_{\ell})^{m(1-\tilde{\alpha})}\Big(\frac{\beta^{j,q}_{\ell}}{\beta^{a,j}_{\ell}}\Big)^{m} \exp\{-N\epsilon\beta^{a,j}_{\ell}(1-\tilde{\alpha})\}.
\end{align}
Moreover the function $v(x)=x^{m(1-\tilde{\alpha})}\exp\{-2x(1-\tilde{\alpha})\}$ reaches its maximum at $x=m/2$ thus we have
\begin{equation*}
  x^{m(1-\tilde{\alpha})}\exp\{-2x(1-\tilde{\alpha})\}\leq\Big(\frac{m}{2}\Big)^{m(1-\tilde{\alpha})}\exp\{-m(1-\tilde{\alpha})\}\quad\forall x
\end{equation*}
In particular
\begin{equation*}
  (N\epsilon\beta^{a,j}_{\ell})^{m(1-\tilde{\alpha})} \exp\{-2N\epsilon\beta^{a,j}_{\ell}(1-\tilde{\alpha})\}\leq\Big(\frac{m}{2}\Big)^{m(1-\tilde{\alpha})}\exp\{-m(1-\tilde{\alpha})\}.
\end{equation*}
Thus
\begin{align}\label{Inequal2}
&\sum_{m\geq0}\frac{m^{\tilde{\alpha}m}\exp\{-\tilde{\alpha}m\}}{m!} (N\epsilon\beta^{a,j}_{\ell})^{m(1-\tilde{\alpha})}\Big(\frac{\beta^{j,q}_{\ell}}{\beta^{a,j}_{\ell}}\Big)^{m} \exp\{-N\epsilon\beta^{a,j}_{\ell}(1-\tilde{\alpha})\} \nonumber \\
&\leq \exp\{N\epsilon\beta^{a,j}_{\ell}(1-\tilde{\alpha})\}\sum_{m\geq0}\frac{m^{m} \exp\{-m\}}{m!}\Big(\frac{\beta^{j,q}_{\ell}/\beta^{a,j}_{\ell}}{2^{(1-\tilde{\alpha})}}\Big)^{m}
\end{align}
Moreover for $q=-$ we have 
\begin{equation*}
 \frac{\beta^{j,-}_{\ell}}{\beta^{a,j}_{\ell}}\leq\frac{\beta_{j}(Z^{N}(\ell\epsilon))}{\beta_{j}(Z^{N,a}(\ell\epsilon))-Cdg(\epsilon)}
 \end{equation*}
If $\beta_{j}(Z^{N}(\ell\epsilon))<\lambda_{1}$ we have using the assumptions \ref{assump1} \ref{assump13} and  \ref{assump1} \ref{assump14}
\begin{align*}
 \frac{\beta^{j,-}_{\ell}}{\beta^{a,j}_{\ell}}&\leq\frac{\beta_{j}(Z^{N,a}(\ell\epsilon))}{\beta_{j}(Z^{N,a}(\ell\epsilon))-Cdg(\epsilon)}\leq\frac{C_{a}}{C_{a}-Cdg(\epsilon)}\\
 &\leq\frac{1}{1-\frac{Cdg(\epsilon)}{g^{1/2}(\epsilon)}}\rightarrow1\quad\text{as}\quad\epsilon\rightarrow0.
 \end{align*}
If $\beta_{j}(Z^{N}(\ell\epsilon))\geq\lambda_{1}$, we have
\begin{align*}
 \frac{\beta^{j,-}_{\ell}}{\beta^{a,j}_{\ell}}&\leq\frac{\beta_{j}(Z^{N}(\ell\epsilon))}{\beta_{j}(Z^{N}(\ell\epsilon))-C\bar{C}a-Cdg(\epsilon)}\leq\frac{\lambda_{1}}{\lambda_{1}-C\bar{C}h(\epsilon)-Cdg(\epsilon)}\\
 &\rightarrow1\quad\text{as}\quad\epsilon\rightarrow0.
 \end{align*} 
And for $q=+$ We have
\begin{equation*}
 \frac{\beta^{j,+}_{\ell}}{\beta^{a,j}_{\ell}}\leq\frac{\beta_{j}(Z^{N}(\ell\epsilon))+Cdg(\epsilon)}{\beta_{j}(Z^{N,a}(\ell\epsilon))-Cdg(\epsilon)}
 \end{equation*}
If $\beta_{j}(Z^{N}(p\epsilon))<\lambda_{1}$ we have using the assumptions \ref{assump1} \ref{assump13}  and \ref{assump1} \ref{assump14}
\begin{align*}
 \frac{\beta^{j,+}_{\ell}}{\beta^{a,j}_{\ell}}&\leq\frac{\beta_{j}(Z^{N,a}(\ell\epsilon))+Cdg(\epsilon)}{\beta_{j}(Z^{N,a}(\ell\epsilon))-Cdg(\epsilon)} \\
 &\leq\frac{C_{a}+Cdg(\epsilon)}{C_{a}-Cdg(\epsilon)}\leq\frac{1+\frac{Cdg(\epsilon)}{g^{1/2}(\epsilon)}}{1-\frac{Cdg(\epsilon)}{g^{1/2}(\epsilon)}}\rightarrow1\quad{as}\quad\epsilon\rightarrow0.
 \end{align*}
If $\beta_{j}(Z^{N}(\ell\epsilon))\geq\lambda_{1}$, we have
\begin{align*}
 \frac{\beta^{j,+}_{\ell}}{\beta^{a,j}_{\ell}}&\leq\frac{\beta_{j}(Z^{N}(\ell\epsilon))+Cdg(\epsilon)}{\beta_{j}(Z^{N}(\ell\epsilon))-C\bar{C}h(\epsilon)-Cdg(\epsilon)}\\
 &\leq\frac{\lambda_{1}+Cdg(\epsilon)}{\lambda_{1}-C\bar{C}h(\epsilon)-Cdg(\epsilon)}\rightarrow1\quad\text{as}\quad\epsilon\rightarrow0.
 \end{align*}
Then there exists $\epsilon_{\alpha}$ such that $\frac{\beta_{\ell}^{j,q}}{\beta_{\ell}^{a,j}}<2^{(1-\alpha)/2}<2^{(1-\tilde{\alpha})/2}$ for all $\epsilon<\epsilon_{\alpha}$.
\\ Thus for $\epsilon$ small enough we have
\begin{align}\label{Inequal3}
&\exp\{N\epsilon\beta_{\ell}^{a,j}(1-\tilde{\alpha})\}\sum_{m\geq0}\frac{m^{m}e^{-m}}{m!}\Big(\frac{\beta_{\ell}^{j,q}/\beta_{\ell}^{a,j}}{2^{(1-\tilde{\alpha})}}\Big)^{m} \nonumber\\ 
&\leq e^{N\epsilon\theta(1-\tilde{\alpha})}\sum_{m\geq0}\frac{m^{m}e^{-m}}{m!}\Big(\frac{1}{2^{(1-\alpha)/2}}\Big)^{m}\\
&= e^{N\epsilon\theta(1-\tilde{\alpha})} K_{\alpha}.\nonumber
\end{align}
Since the series above converges. We deduce from \eqref{Inequal1}, \eqref{Inequal2} and \eqref{Inequal3} that
\begin{align*}
 \mathbb{E}_{z}\Big(\exp\Big\{\alpha N\epsilon f\Big(\frac{(1-a)Z^{q}_{j}}{\epsilon N}, \beta^{a,j}_{\ell}\Big)\Big\}|\mathcal{F}^{N}_{\ell\epsilon}\Big)&\leq K_{\alpha}\exp\{N\epsilon C_{2}(1-\alpha+a)\}\exp\{N\epsilon\tilde{C}(a+cdg(\epsilon))\}\\
 &\leq K_{\alpha}\exp\{N\epsilon\tilde{K}(1-\alpha+2h(\epsilon)+2dg(\epsilon))\}.
 \end{align*} 
\end{proof}

Thus, we have
\begin{equation*}
 \mathbb{E}_{z}(\Theta_{j}^{\ell}|\mathcal{F}^{N}_{\ell\epsilon})\leq\mathbb{E}_{z}(\Xi_{j}^{\ell}|\mathcal{F}{N}_{\ell\epsilon})\leq2K_{\alpha}\exp\{N\epsilon\tilde{K}_{1}(1-\alpha+2h(\epsilon)+4dg(\epsilon))\}.
\end{equation*}

The next lemma gives us a upper bound for the quantity $\mathbb{E}_{z}\Big(\exp\{\alpha NI_{T}(\Upsilon|\mu)\}\mathbf{1}_{B_{\epsilon}}\Big)$.
\begin{lemma}
We have the following inequality
\begin{equation}\label{ineqsup1}
\mathbb{E}_{z}\Big(\exp\{\alpha N I_{T}(\Upsilon|\mu)\}\mathbf{1}_{B_{\epsilon}}\Big)\leq (2K_{\alpha})^{\frac{kT}{\epsilon}}\exp\{kNT\tilde{K}_{1}(1-\alpha+h(\epsilon)+4dg(\epsilon))\}
\end{equation}
\end{lemma}
\begin{proof}
We know that $\Xi^{\ell}_{j}$, $j=1,...,k$ are conditionnally independent given $\mathcal{F}^{N}_{\ell\epsilon}$. Taking iterative conditional expectations with 
respect to  $\mathcal{F}^{N}_{(\frac{T}{\epsilon}-1)\epsilon}$, $\mathcal{F}^{N}_{(\frac{T}{\epsilon}-2)\epsilon}$,...,$\mathcal{F}^{N}_{\epsilon}$, we get that for all $0<\alpha<1$ and $\epsilon<\epsilon_{\alpha}$
\begin{align*}
\mathbb{E}_{z}\Big(\exp\{\alpha NI_{T}(\Upsilon|\mu)\}\mathbf{1}_{B_{\epsilon}}\Big)&=\mathbb{E}_{z}\Big(\prod_{\ell=0}^{\frac{T}{\epsilon}-1}\exp\Big\{\alpha N\int_{\ell\epsilon}^{(\ell+1)\epsilon}\sum_{j}f(\mu^{j}_{t},\beta_{j}(\Upsilon_{t}))dt\Big\}\mathbf{1}_{B^{\ell}_{\epsilon}}\Big)\\
&=\mathbb{E}_{z}\Big(\mathbb{E}_{z}\Big(\prod_{\ell=0}^{\frac{T}{\epsilon}-1}\prod_{j=1}^{k}\Theta_{j}^{\ell}|\mathcal{F}^{N}_{(\frac{T}{\epsilon}-1)\epsilon}\Big)\Big)\leq\mathbb{E}^{N}\Big(\mathbb{E}_{z}\Big(\prod_{\ell=0}^{\frac{T}{\epsilon}-1}\prod_{j=1}^{k}\Xi_{j}^{\ell}|\mathcal{F}^{N}_{(\frac{T}{\epsilon}-1)\epsilon}\Big)\Big)\\
&\leq\mathbb{E}_{z}\Big(\prod_{\ell=0}^{\frac{T}{\epsilon}-2}\prod_{j=1}^{k}\Xi_{j}^{\ell}\mathbb{E}^{N}\Big(\prod_{j=1}^{k}\Xi_{j}^{\frac{T}{\epsilon}-1}|\mathcal{F}^{N}_{(\frac{T}{\epsilon}-1)\epsilon}\Big)\Big)\\
&\leq\prod_{p=0}^{\frac{T}{\epsilon}-1}(2K_{\alpha})^{k}\exp\{kN\epsilon\tilde{\tilde{C}}(1-\alpha+h(\epsilon)+4dg(\epsilon))\}\\
&=(2K_{\alpha})^{\frac{kT}{\epsilon}}\exp\{kNT\tilde{K}_{1}(1-\alpha+h(\epsilon)+4dg(\epsilon))\}.
\end{align*}
\end{proof}
In the next Lemma, we give a upper bound for $\mathbb{P}_{z}(B_{\epsilon}^{c})$.
\begin{lemma}
For any $s>0$ there exists $\epsilon_{0}>0$, $N_{0}\in\mathbb{N}$ and $K>0$ such that
\begin{equation}\label{ineqsup2} 
\mathbb{P}_{z}(B_{\epsilon}^{c})<\frac{dkT}{\epsilon}\exp\{-sN\} 
\end{equation}
for all $\epsilon<\epsilon_{0}$ and $N>N_{0}$ where $g(\epsilon)=K\sqrt{\log^{-1}(\epsilon^{-1})}$. 
\end{lemma}
\begin{proof}
 For all $j=1,...,k$ and $\ell=1,...,T/\epsilon$ we can write
\begin{equation*}
 \int_{0}^{(\ell+1)\epsilon}\beta_{j}(Z^{N}_{s})ds<\int_{0}^{\ell\epsilon}\beta_{j}(Z^{N}_{s})ds+\sigma\epsilon.
\end{equation*}
Moreover, we have
\begin{equation*}
 B_{\epsilon}^{c}=\bigcup_{i=1,...,d}\bigcup_{\ell=1,...,T/\epsilon}\Big\{\sup_{(\ell-1)\epsilon\leq t_{1},t_{2}\leq \ell\epsilon}|Z^{N}_{i}(t_{1})-Z^{N}_{i}(t_{2})|>g(\epsilon)\Big\}.
\end{equation*}
Thus
\begin{equation*}
 \mathbb{P}_{z}(B_{\epsilon}^{c})\leq\sum_{i=1}^{d}\sum_{\ell=1}^{T/\epsilon}\mathbb{P}\Big\{\sup_{(\ell-1)\epsilon\leq t_{1},t_{2}\leq \ell\epsilon}|Z^{N}_{i}(t_{1})-Z^{N}_{i}(t_{2})|>g(\epsilon)\Big\}.
\end{equation*}
Using \eqref{EqPoisson1} and denoting by $Z^{N}_{i}(.)$ the $i^{th}$ coordinate of $Z^{N}(.)$ we have, since $|h_{j}^{i}|\leq1$ for all $1\leq j\leq k$, $1\leq i\leq d$,
\begin{align*}
& \sup_{(\ell-1)\epsilon\leq t_{1},t_{2}\leq \ell\epsilon}|Z^{N}_{i}(t_{1})-Z^{N}_{i}(t_{2})| \\
 &=\sup_{(\ell-1)\epsilon\leq t_{1},t_{2}\leq \ell\epsilon}\Big|\sum_{j}\frac{h^{i}_{j}}{N}\Big[P_{j}\Big(N\int_{0}^{t_{1}}\beta_{j}(Z^{N}(s))ds\Big)-P_{j}\Big(N\int_{0}^{t_{2}}\beta_{j}(Z^{N}(s))ds\Big)\Big]\Big|\\
 &\leq\frac{1}{N}\sum_{j}\Big[P_{j}\Big(N\int_{0}^{\ell\epsilon}\beta_{j}(Z^{N}(s))ds\Big)-P_{j}\Big(N\int_{0}^{(\ell-1)\epsilon}\beta_{j}(Z^{N}(s))ds\Big)\Big]\\
 &\leq\frac{1}{N}\sum_{j}\Big[P_{j}\Big(N\int_{0}^{(\ell-1)\epsilon}\beta_{j}(Z^{N}(s))ds+N\sigma\epsilon\Big)-P_{j}\Big(N\int_{0}^{(\ell-1)\epsilon}\beta_{j}(Z^{N}(s))ds\Big)\Big]\\
 &\leq\frac{1}{N}\sum_{j}Z_{j}.
\end{align*}
Where $Z_{j}$  $j=1,...,k$ are independent Poisson random variables with mean $N\sigma\epsilon$.
Then 
\begin{equation*}
 \mathbb{P}_{z}\Big\{\sup_{(\ell-1)\epsilon\leq t_{1},t_{2}\leq \ell\epsilon}|Z^{N}_{i}(t_{1})-Z^{N}_{i}(t_{2})|>g(\epsilon)\Big\}\leq k\mathbb{P}_{z}(N^{-1}Z_{1}>g(\epsilon)/k)
\end{equation*}
And it follows from  lemma \ref{le17} that there exist a constants $K>0$, $\epsilon_{0}>0$ and $N_{0}\in\mathbb{N}$ such that
\begin{equation*}
 \mathbb{P}_{z}\Big\{\sup_{(\ell-1)\epsilon\leq t_{1},t_{2}\leq \ell\epsilon}|Z^{N}_{i}(t_{1})-Z^{N}_{i}(t_{2})|>g(\epsilon)\Big\}\leq k\exp\{-sN\}
\end{equation*}
For all $\epsilon<\epsilon_{0}$ and $N>N_{0}$. And then
\begin{equation*}
\mathbb{P}_{z}(B_{\epsilon}^{c})<\frac{dkT}{\epsilon}\exp\{-sN\}. 
\end{equation*}
\end{proof}

\par Now, we find a upper bound for $\mathbb{P}_{z}(\|Z^{N,a}-\Upsilon\|_{T}\geq\delta/2d)$ in \eqref{major2}.
\begin{lemma}
 For all $\delta, s>0$ there exist $\epsilon_{\alpha}>0$, $N_{0}\in\mathbb{N}$ such that
\begin{equation}\label{ineqsup3}
\mathbb{P}_{z}(\|Z^{N}_{a}-\Upsilon\|_{T}>\delta/2d)<\frac{dkT}{\epsilon}\exp\{-sN\}, 
\end{equation}
for all $\epsilon<\epsilon_{\alpha}$ and $N>N_{0}$.
\end{lemma}
\begin{proof}
Using \eqref{EqPoisson1} we write for all $t\in[\ell\epsilon,(\ell+1)\epsilon[$
\begin{align*}\label{major3}
 |Z^{N}_{a,i}(t)-\Upsilon^{i}_{t}|&\leq\sum_{j}\frac{1}{N}\Big[P_{j}\Big(N\int_{0}^{(\ell+1)\epsilon}\beta_{j}(Z^{N}(s))ds\Big)-P_{j}\Big(N\int_{0}^{\ell\epsilon}\beta_{j}(Z^{N}(s))ds\Big)\Big]\\
 &\leq\frac{1}{N}\sum_{j}\Big[P_{j}\Big(N\int_{0}^{\ell\epsilon}\beta_{j}(Z^{N}(s))ds+N\sigma\epsilon\Big)-P_{j}\Big(N\int_{0}^{\ell\epsilon}\beta_{j}(Z^{N}(s))ds\Big)\Big]\\
 &\leq\frac{1}{N}\sum_{j}Z_{j}
\end{align*}
where the $Z_{j}$ are as in the proof of the last lemma. Let $\epsilon_{1}$ be the maximal $\epsilon$ such that $\delta/2kd^{2}>g(\epsilon)$. Then we have from lemma \ref{le17} that
for all $\epsilon<\epsilon_{\alpha}=\min\{\epsilon_{0},\epsilon_{1}\}$ and $N>N_{0}$
\begin{align*}
 \mathbb{P}_{z}(\|Z^{N}_{a}-\Upsilon\|_{T}>\delta/2d)&\leq\mathbb{P}_{z}\Big(\bigcup_{i=1}^{d}\{|Z^{N}_{a,i}(t)-\Upsilon^{i}_{t}|>\frac{\delta}{2d^2}\}\quad\text{for some}\quad t\in[0,T]\Big)\\
 & \leq \frac{T}{\epsilon}\max_{0\leq \ell\leq T/\epsilon-1}\mathbb{P}_{z}\Big(\bigcup_{i=1}^{d}\{|Z^{N}_{a,i}(t)-\Upsilon^{i}_{t}|>\frac{\delta}{2d^2}\}\quad\text{for some}\quad t\in[\ell\epsilon,(\ell+1)\epsilon[\Big)\\
 &\leq \frac{dkT}{\epsilon}\mathbb{P}_{z}(Z_{1}/N>\delta/2kd^2)\leq \frac{dkT}{\epsilon} \exp\{-sN\}.
\end{align*}
\end{proof}

The end of the proof of the lemma \ref{upperb} can be done by using \eqref{ineqsup1}, \eqref{ineqsup2}, \eqref{ineqsup3}.
 We have thus for all $\delta>0$, $0<\alpha<1$, $\epsilon<\min\{\epsilon_{0},\epsilon_{\frac{\delta}{2d}},\epsilon_{1}\}$ and $a=h(\epsilon)=\Big[-\log g^{1/2}(\epsilon)\Big]^{-\frac{1}{\nu}}$,
\begin{align*}
 \mathbb{P}_{z}(\rho_{T}(Z^{N},\Phi(s))\geq\delta)&\leq\mathbb{P}_{z}(I_{T}(\Upsilon|\mu)\geq s)+\mathbb{P}(\|Z^{N}_{a}-\Upsilon\|_{T}\geq\delta/d)\\
 &\leq\frac{\mathbb{E}_{z}(\exp\{\alpha N I_{T}(\Upsilon|\mu)\}\mathfrak{1}_{B_{\epsilon}})}{\exp\{\alpha N s\}}+\mathbb{P}_{z}(B_{\epsilon}^{c})+\mathbb{P}_{z}(\|Z^{N}_{a}-\Upsilon\|_{T}\geq\delta/2d)\\
 &\leq (2K_{\alpha})^{\frac{kT}{\epsilon}}\exp\{kNT\tilde{K}_{1}(1-\alpha+h(\epsilon)+4dg(\epsilon))\}\\
 &\times\exp\{-\alpha N s\}+\frac{2dTk}{\epsilon}\exp\{-sN\}.
\end{align*}
Here, we take $1-\alpha$ and $\epsilon$ small enough to ensure that $kT\tilde{K}_{1}(1-\alpha+h(\epsilon)+4dg(\epsilon))<\eta/4$
and $(1-\alpha)s<\eta/4$. We also take $N$ large enough so that $kT\log(2K_{\alpha})/N\epsilon<\eta/4$ and $\log(2dkT/\epsilon)/N<\eta/4$ and we have
\begin{align*}
 \mathbb{P}_{z}(\rho_{T}(Z^{N},\Phi(s))\geq\delta)&\leq\exp\{-N(s-3\eta/4)\}+\frac{2dT}{\epsilon}\exp\{-sN\}\\
 &\leq \frac{dkT}{\epsilon}.\exp\{-N(s-3\eta/4)\}\leq\exp\{-N(s-\eta)\}.
\end{align*}
Thus
\begin{equation*}
\mathbb{P}_{z}^{N}(F_{\delta}^{s})\leq\exp\{-N(s-\eta)\}.
\end{equation*}
\end{proof}

\par We conclude the proof of the upper bound by the following theorem
\begin{theorem}\label{upperbound1}
For any closed subset $F$ of $D_{T,A}$ and $z\in A$
\begin{equation}\label{upper}
\limsup_{N\to\infty}\frac{1}{N}\log\mathbb{P}^{N}_{z}(F)\leq -\inf_{\phi\in F,\phi_{0}=z}I_{T}(\phi).
\end{equation}
\end{theorem}
\begin{proof}
Show that if the inequality  \eqref{upper} is true then the inequality  \eqref{loexpupper} is also true. To this end, we remark that for all $\delta$ and $s>0$, $F_{\delta}^{s}$ defined by \eqref{hset} is closed and $I_{T}(\phi)>s$ for all $\phi\in F_{\delta}^{s}$. Therefore $\inf\{I_{T}(\phi): \phi\in F_{\delta}^{s}, \phi_{0}=z\}\geq\inf\{I_{T}(\phi): \phi\in F_{\delta}^{s}\}\geq s$. We deduce from inequality \eqref{upper} that
\begin{equation*}
\limsup_{N\to\infty}\frac{1}{N}\log\mathbb{P}_{z}^{N}(F_{\delta}^{s})\leq -s.
\end{equation*}
Then for all $\eta>0$ there exists $N_{0}\in\mathbb{N}$ such that for all $N>N_{0}$ we have
\begin{equation*}
\mathbb{P}_{z}^{N}(F_{\delta}^{s})\leq \exp\{-N(s-\eta)\}.
\end{equation*}

We now assume that the inequality \eqref{loexpupper} is satisfied and we need to prove that this implies  \eqref{upper}. To this end let $F\in D_{T,A}$ a closed set,  choose $\eta>0$ and let $$s=\inf_{\phi\in F, \phi_{0}=z}I_{T}(\phi)-\eta/2.$$ The closed set $F_{z}=\{\phi\in F: \phi_{0}=z\}$ does not intersect the compact set $\Phi(s)$. Therefore $$\delta=\inf_{\phi\in F_{z}}\inf_{\psi\in\Phi(s)}\|\phi-\psi\|_{T}>0.$$ We use the inequality  \eqref{loexpupper}  to have for any $\delta, \eta$ and $s>0$ there exists $N_{0}\in\mathbb{N}$ such that for all $N>N_{0}$,
\begin{align*}
\mathbb{P}_{z}^{N}(F)&\leq\mathbb{P}_{z}^{N}(F_{\delta}^{s})\\
&\leq \exp\{-N(s-\eta/2)\} \\
&\leq \exp\Big\{-N\Big(\inf_{\phi\in F,\phi_{0}=z}I_{T}(\phi)-\eta\Big)\Big\},
\end{align*}
then 
\begin{equation*}
\limsup_{N\to\infty}\frac{1}{N}\log\mathbb{P}_{z}^{N}(F)\leq -\inf_{\phi\in F,\phi_{0}=z}I_{T}(\phi).
\end{equation*}
\end{proof}

\begin{corollary}
For any open subset $F$ of $D_{T,A}$ and any compact subset $K$ of $A$,
  \begin{equation*}
\limsup_{N\to\infty}\frac{1}{N}\log\sup_{z\in K}\mathbb{P}_{z}(Z^{N}\in F)\leq-\inf_{z\in K}\inf_{\phi\in F, \phi_{0}=z}I_{T}(\phi).
\end{equation*}
\end{corollary}

\begin{proof}
The arguments are the same as in the proof of Corollary 5.6.15 in \cite{Dembo2009}. Let
\begin{equation*}
I_{K}:=\inf_{z\in K}\inf_{\phi\in F, \phi_{0}=z}I_{T}(\phi).
\end{equation*}
For  $\eta>0$ fix, let $I_{K}^{\eta}:=\min\{I_{K}-\eta, \eta^{-1}\}$. Then from Lemma \ref{semiconz} it follows that for any $z\in K$, there exists a $N_{z}$ such that for all $N>N_{z}$ and $y\in B(z,\frac{1}{N_{z}})$,
\begin{equation*}
\inf_{\phi\in F, \phi_{0}=y}I_{T}(\phi)\geq\inf_{\phi\in F, \phi_{0}=z}I_{T}(\phi)-\eta\geq I_{K}^{\eta}.
\end{equation*}
Therefore we have from \eqref{upper} that 
\begin{equation*}
\frac{1}{N}\log\mathbb{P}_{y}(Z^{N}\in F)\leq-\inf_{\phi\in F, \phi_{0}=y}I_{T}(\phi)\leq-I_{K}^{\eta}.
\end{equation*}
And then
\begin{equation*}
\frac{1}{N}\log\sup_{y\in B(z,\frac{1}{N_{z}})}\mathbb{P}_{y}(Z^{N}\in F)\leq-I_{K}^{\eta}.
\end{equation*}
As $K$ is compact, there exits a finite sequence $(z_{i})_{1\leq i\leq m}\subset K$  such that  $K\subset\bigcup_{i=1}^{m}B(z_{i},\frac{1}{N_{z_{i}}})$. Then for $N\geq \max_{1\leq i\leq m}N_{z_{i}}$,
\begin{equation*}
\frac{1}{N}\log\sup_{y\in K}\mathbb{P}_{y}(Z^{N}\in F)\leq-I_{K}^{\eta}.
\end{equation*}
It first remains to take $\limsup$ as $N\to\infty$ and then let $\eta$ tend to $0$ to have result.
\end{proof}

\section{Time of exit from a domain}
Let $O$ the domain of attraction of a stable point of the dynamical system \eqref{ODE} and $\widetilde{\partial O}$ be the part of boundary of $O$ that the stochastic system \eqref{EqPoisson1} can cross. We now give an approximate value for the exit time $\tau^{N}_{O}$ from $O$ for large $N$ as well as the exponential asymptotic of its mean $\mathbb{E}_{z}(\tau^{N}_{O})$. To this end, for $z, y \in \bar{O}$, we define the following functionals
\begin{align*}
V_{\bar{O}}(z,y,T)&:= \inf_{\phi\in D_{T,\bar{O}}, \phi_{0}=z, \phi_{T}=y} I_{T}(\phi) \\
V_{\bar{O}}(z,y)&:= \inf_{T>0} V_{\bar{O}}(z,y,T)  \\
 V_{\widetilde{\partial O}}&:= \inf_{y \in \widetilde{\partial O}} V_{\bar{O}}(z^{*},y).
\end{align*}
The following theorem is a consequence of the large deviation principle established above, the law of large numbers and some technical arguments. The proof could be found in Section 7 of \cite{Kratz2014}.

\begin{theorem}
Given $\eta>0$, for all $z\in O$,
\begin{equation*}
\lim_{N\to\infty}\mathbb{P}_{z}\big(\exp\{N(V_{\widetilde{\partial O}}-\eta)\}<\tau^{N}_{O}<\exp\{N(V_{\widetilde{\partial O}}+\eta)\}\big)=1.
\end{equation*}
Moreover, for all $\eta>0$, $z\in O$ and $N$ large enough,
\begin{equation*}
\exp\{N(V_{\widetilde{\partial O}}-\eta)\}\leq\mathbb{E}_{z}(\tau^{N}_{O})\leq\exp\{N(V_{\widetilde{\partial O}}+\eta)\}.
\end{equation*}
\end{theorem}

\frenchspacing
\bibliographystyle{plain}

\end{document}